\def\version{29.8.2017 - revised 5.4.2018}\def\users{}  %
\def\users{final-layout}  
\documentclass[12pt]{article}
\usepackage{graphicx}
\usepackage{graphics}
\usepackage{amsfonts}
\usepackage{amssymb}
\usepackage{amsmath}
\usepackage{amsthm}
\usepackage{mathrsfs}
\usepackage{color}
\usepackage{psfrag}     

\usepackage{epsfig}
\usepackage{psfrag}
\usepackage{graphicx}
\usepackage{textcomp}
\usepackage{pgf}   
\usepackage{upgreek} 
\usepackage{mathrsfs,cite}

\newtheorem{theorem}{Theorem}[section]

\newtheorem{lemma}[theorem]{Lemma}
\newtheorem{proposition}[theorem]{Proposition}

\newtheorem{definition}[theorem]{Definition}
\newtheorem{remark}[theorem]{Remark}

\numberwithin{equation}{section}

\usepackage{ifthen}
\ifthenelse{\equal{\users}{final-layout}}{}{
\usepackage{fancyhdr}
\pagestyle{fancy}
\headheight=28pt\headwidth=16.5cm
\definecolor{gray}{gray}{0.5}
\rhead{\color{green}Thermoelastoplasticity at large strains\\
T.Roub\'\i\v cek \& U.Stefanelli}
\chead{}
\lhead{\color{green}Version \version, file: \jobname.tex\\
compiled:
\number\day.\number\month.\number\year\ at
\the\hour:\ifnum\minute<10 0\fi\the\minute\ h }
}

\newcount\hour \newcount\minute
\hour=\time
\divide \hour by 60
\minute=\time
\loop \ifnum \minute > 59 \advance \minute by -60 \repeat

\usepackage[normalem]{ulem}
\definecolor{brown}{rgb}{0.5,0,0}
\ifthenelse{\equal{\users}{final-layout}}{
    \newcommand{\REPLACE}[2]{#2}
    
    \newcommand{\DELETE}[1]{}
    
    \newcommand{\COMMENT}[1]{}
    
    \newcommand{\REM}[1]{\marginpar{\bfseries\tiny{\color{blue}}}}
    \newcommand{\TODO}[1]{}

\newcommand{\UUU}{\color{black}}

\newcommand{\EEE}{\color{black}}
}
{\usepackage[notcite,notref,color]{showkeys}
\definecolor{labelkey}{rgb}{1.,.2,0.}

\newcommand{\REPLACE}[2]{{\color{brown}\sout{#1}\uline{#2}\color{black}}}

 \newcommand{\DELETE}[1]{{\color{brown}\sout{#1}\color{black}}}
 
 \newcommand{\COMMENT}[1]{{\color{red}\uuline{#1}\color{black}}}
 
 \newcommand{\REM}[1]{\marginpar{\bfseries\tiny{\color{blue}#1}}}
 \newcommand{\TODO}[1]{{$^{\color{blue}{TODO:}}$\footnote{\color{blue}{#1}}}}  

\newcommand{\UUU}{\color{blue}}

\newcommand{\EEE}{\color{black}}
}

\newcommand{\R}{\mathbb{R}}

\newcommand{\bbI}{\mathbb{I}}
\newcommand\DT[1]{\mathchoice
                 {{\buildrel{\hspace*{.1em}\text{\LARGE.}}\over{#1}}}
                 {{\buildrel{\hspace*{.1em}\text{\Large.}}\over{#1}}}
                 {{\buildrel{\hspace*{.1em}\text{\large.}}\over{#1}}}
                 {{\buildrel{\hspace*{.1em}\text{\large.}}\over{#1}}}}
\newcommand\DDT[1]{\mathchoice
   {{\buildrel{\hspace*{.13em}\text{\LARGE.\hspace*{-.13em}.}}\over{#1}}}
   {{\buildrel{\hspace*{.1em}\text{\Large.\hspace*{-.1em}.}}\over{#1}}}
   {{\buildrel{\hspace*{.1em}\text{\large.\hspace*{-.1em}.}}\over{#1}}}
   {{\buildrel{\hspace*{.1em}\text{\large.\hspace*{-.1em}.}}\over{#1}}}}

\newcommand{\linesunder}[3]{\LSU{\begin{array}[t]{c}\underbrace{#1}\vspace*{.5em}\end{array}}{\mbox{\footnotesize\rm #2}}{\mbox{\footnotesize\rm#3}}}

\newcommand{\LSU}[3]{\begin{array}[t]{c}#1\vspace*{-1em}\\_{#2}\vspace*{-.3em}\\_{#3}\end{array}}
\newcommand{\Vdots}{\vdots}
\renewcommand{\d}{{\rm d}}

\newcommand\Frakk{\text{\large$\mathbb{K}$}}

\newcommand\FT{\psi_{_{\rm T}}}
\newcommand\FE{\psi_{_{\rm E}}}
\newcommand\FH{\psi_{_{\rm H}}}
\newcommand\widehatFM{\widehat\psi_{_{\rm M}}}
\newcommand\PsiT{\Psi_{_{\rm T}}}
\newcommand\PsiM{\Psi_{_{\rm M}}}
\newcommand\PP{P}
\newcommand\PR{R}
\newcommand\Cof{\mathrm{Cof}}

\newcommand{\divS}{\mathrm{div}_{\scriptscriptstyle\textrm{\hspace*{-.1em}S}}^{}}

\begin{document}


\vspace*{2em}

\noindent{\LARGE\bf
Finite thermoelastoplasticity and creep
\\
under small elastic strains}

\bigskip\bigskip

\noindent{\large\sc Tom\'{a}\v{s} Roub\'\i\v{c}ek}\\
{\it Mathematical Institute, Charles University, \\Sokolovsk\'a 83,
CZ--186~75~Praha~8,  Czech Republic
}\\and\\
{\it Institute of Thermomechanics, Czech Academy of Sciences,\\Dolej\v skova~5,
CZ--182~08 Praha 8,  Czech Republic
}

\bigskip

\noindent{\large\sc Ulisse Stefanelli}\\
{\it Faculty of Mathematics, University of
    Vienna, \\Oskar-Morgenstern-Platz 1, A-1090 Vienna, Austria}\\ and \\
{\it Istituto di Matematica Applicata e Tecnologie Informatiche
  E. Magenes - CNR, \\v. Ferrata 1, I-27100 Pavia, Italy}


\begin{center}\begin{minipage}[t]{16.5cm}

{\small \noindent {\it Abstract.} \baselineskip=12pt
A mathematical model for an elastoplastic continuum subject to large 
strains is presented. The inelastic response is modeled within 
the frame of rate-dependent gradient plasticity for nonsimple materials. 
Heat diffuses through the continuum by the Fourier law in the actual 
deformed configuration. Inertia makes the nonlinear problem
hyperbolic. \UUU The modelling assumption of small elastic Green-Lagrange 
strains \EEE is combined in a \UUU thermodynamically \EEE consistent way with 
the possibly large displacements and large plastic strain. The model is 
amenable to a rigorous mathematical analysis. \UUU The existence \EEE
of suitably defined weak solutions and a convergence result for 
Galerkin approximations is proved.
\medskip


\medskip

}
\end{minipage}
\end{center}

\vspace{1.5em}

\section{Introduction} 
  
{\it Thermoelastodynamic problems} in combination with
inelastic effects  are ubiquitous in
applications and have triggered an intense research activity cutting
across material science, engineering, and mathematics \cite{JirBaz02IAS,Maug92TPF,SimHug98CI}. A number of
rigorous mathematical results are available in case of small strains,
whereas the literature for finite-strain problems is relatively less
developed. Still, large strains arise naturally in a variety of
different thermomechanical contexts including, for instance, plastic
deformations, rolling, and
impacts, and in a number of different materials, from metals to polymers.


The focus of this note is to contribute a finite-strain
thermomechanical model for an inelastic continuum.  The evolution of
the medium will be described by its deformation $y$ from the reference
configuration, its plastic strain $P$, and its absolute temperature
$\theta$. We consider a fully dynamic problem and address  viscoelastic rheologies of {\it Maxwell}
type,  including  {\it creep} \cite{PleKor02EIMT,RoyRed14CMPC} or
({\it visco}){\it plasticity} \cite{MiRoSa??GERV}. \UUU These are both
permanent deformation dynamics. Plasticity is an activated effect, for
its onset corresponds to the reaching of a given yield stress. On the
contrary, one usually
speaks of creep in case 
material relazation is  always 
active upon loading, without any yield threshold
\cite{BJnew}. In this paper, creep is modeled by the Maxwell
viscoelastic rheology; note however that different rheologies are sometimes
associated in the literature with creep as well.
\EEE
Within the classical frame of Generalized Standard Materials 
\cite{Germain,Halphen}, the model results
from the combination of energy and momentum conservation with the
plastic flow rule and will be checked to be thermodynamically consistent.

Our crucial modeling tenet is that
the large deformations are stored by the plastic strain $P$, whereas
the {\it elastic part} of the total deformation strain remains close
to the identity. This assumptions is often not restrictive, as elastoplastic
materials often sustain relatively small elastic deformations before
plasticizing. This is for instance the case of ordinary metals, which
usually plasticize around a few strain percents, as well as of 
geophysical applications,  where soils and rocks sustain small
elastic deformations before sliding and cracking  \cite{LyHaBZ11NLVE}.
On the other hand, such smallness assumption allows for a satisfactory
mathematical treatment in terms  existence and approximability of
solutions. 

Note that we explicitly include in the model nonlocal effects in
terms of higher order gradients of $\nabla y$ and $P$. This sets our
model within the frame of {\it 2nd-grade
  nonsimple} \cite{Toup62EMCS}  and  {\it gradient
  plastic} materials \cite{Aifantis91}. \REPLACE{On the other hand, we consider
{\it perfect isochoric} plastic dynamics. In particular, no hardening
is assumed.
 
This simplifies the notation and widens the applicability of
the model  at the expense of making the analysis more
delicate. Note that including hardening would require just notational
changes.}{We 
 allow for no 
hardening in the shear deformation, which 
is a typical attribute in particular of creep, i.e.\ the
 viscoelastic  rheology of Maxwell type.}

The elastoplastic evolution of the medium is combined with heat production 
and thus heat transfer. A  weak thermal coupling through
the   temperature-dependence of the dissipation potential 
is considered. This models the possible  temperature dependence of the plastic yield 
stress or the viscous moduli responsible for creep.  An important 
feature of the model  is that 
heat convection governed by the {\it Fourier law}  \UUU occurs \EEE  {\it 
in the actual deformed configuration}.  This is  in most applications much 
more physical than  considering heat conduction  in the
material reference configuration,  especially at large
strains. Here again the smallness assumption on the elastic strain
plays a crucial role in avoiding  
analytical technicalities related with the control of the inverse of 
deformation gradient, cf.\   \eqref{M-pull-back} vs. 
\eqref{M-pull-back+} below.

Existence results in finite-strain elastoplasticity are at the moment
restricted to the isothermal case. In the incremental setting, the early 
result in \cite{Miel03EFME} has
been extended in \cite{Mie-Mull06} by including a term in ${\rm curl}\,P$ 
in the energy, see also \cite{compos} for the case of compatible
plastic deformations. In the quasistatic case, the existence of {\it
  energetic solutions} \cite{MieRou15RIST} has been proved in
\cite{MaiMie09GERI} in the frame of gradient plasticity, see also
\cite{cplas_part2,Miel02FELG,Miel04NAEP}. A
discussion on a possible finite-element approximation is in
\cite{MieRou16RIEF} and rigorous linearization results are in \cite{Giacomini13,MieSte13LPEG}.
Viscoplasticity is addressed in \cite{MiRoSa??GERV} instead.

The novelty of this paper is that of dealing with the {\it nonisothermal}
and {\it dynamic}
case. To the best of our knowledge, the analysis of {\it both} these
features is unprecedented in the frame of finite
elastoplasticity. Our main result, Theorem \ref{thm}, states the existence of suitably
weak solutions. This relies on the  smallness assumption on the
elastic strain, which in turn allows us to tame the nonlinear nature
of the coupling of thermal and mechanical effects.
The existence proof is based on a
regularization and Galerkin approximation procedure, which could serve
as basis for numerical investigation.

The plan of the paper is as follows.  In
Section~\ref{sec-model} we formulate the model. In particular,
we specify the form of the total energy and of the dissipation. 
This brings to the formulation of an evolution system of partial differential 
equations and inclusions. The \UUU thermodynamical \EEE consistency of the model
 and various possible modifications are also discussed.  Section~\ref{sec-anal}
presents a variational notion of solution as well as the main
analytical statement. The existence proof is then detailed   
in Section~\ref{sec-proof}.


\section{The model and its thermodynamics}\label{sec-model}
We devote this section to \UUU presenting \EEE our  model
for elastoplastic continua with   heat transfer. This
is formulated in Lagrangian  coordinates with $\Omega\subset\R^d$ ($d=2$ or  
$ 3$)
being a bounded smooth reference (fixed)  configuration. 
 The  variables of the model are
\begin{align*}
y&:\Omega \to \R^d \quad&&\text{deformation,}\\ 
\PP&:\Omega \to 
{\rm GL}^+(d):=\{P\in\R^{d\times d};\ \det P>0\}
\quad&&\text{plastic part of the inelastic
strain,}\\
\theta&:\Omega \to (0,\infty)\quad&&\text{absolute 
temperature. } 
\end{align*}
%

The model will result by combining momentum and energy
conservation with the dynamics of internal variables. In order to
specify the latter and provide constitutive relations, we introduce a
free energy and a dissipation (pseudo)potential in the following subsections.

We consider the standard multiplicative decomposition 
\cite{Kron60AKVE,LeeLiu67FSEP} 
\begin{align}\label{split}
F=F_{\rm el}\PP \ \ \text{ with }F=\nabla y. 
\end{align}
Our basic modelling assumption is that the {\it elastic} part of the
{\it strain is small} in the sense 
\UUU 
that 
\EEE the elastic Green-Lagrange strain 
$E_{\rm el}:=\frac12(F_{\rm el}^\top F_{\rm el}-\mathbb I)$ is
small \UUU is small as well. \EEE  Here,  the 
superscript $\top$ stands for transposition and $\mathbb{I}$ is the identity 
matrix. \UUU Note that $E_{\rm el}$ is sometimes 
called {\it Green-St.\,Venant} strain as well.
In fact, the smallness of $E_{\rm el}$ is equivalent to $F^\top F\sim\PP^\top\PP$
and is weaker than $F\sim\PP$, since it allows large elastic
rotations, a circumstance which may be relevant   in many applications.
\UUU Such a smallness \EEE assumption is well-fitted to the case of
metals  or rocks, where comparably small elastic strains already 
activate the 
 plastification, or to polymers or soils undergoing considerable
 deformation through creep. \UUU Moreover, the \EEE smallness assumption 
allows us to simplify the mathematical treatment of the model. 

Most notably, 
the {\it elastic energy density} $\FE(F_{\rm el})$ of the medium can be assumed 
to have a controlled polynomial growth, see \eqref{ass-FM} below. In addition,
we assume   $\FE$ to be   {\it frame-indifferent}, namely 
\begin{align}\label{frame-indif} 
\forall Q\in{\rm SO}(d):\qquad
\FE(QF_{\rm el})=\FE(F_{\rm el}).
\end{align} 
 Here, we used the notation  ${\rm SO}(d)$ for the matrix 
special-orthogonal group  ${\rm SO}(d):=\{
Q\in
{\rm GL}^+(d);\ 
QQ^\top\!=Q^\top\! Q=\mathbb{I}
\}$. 
Equivalently, $\FE(F_{\rm el})$ can be expressed as a function of the 
elastic Cauchy-Green tensor \UUU $F_{\rm el}^\top F_{\rm el}$ \EEE only. 

The multiplicative decomposition \eqref{split} allows \UUU us \EEE to express
the free energy in terms of the total strain tensor  $\nabla y $ 
and  the plastic strain $\PP$  via the substitution 
\REPLACE{$F_{\rm el}=F\PP^{-1}=F\Cof\PP$, where we
have used the classical definition for the {\it cofactor matrix}
$\Cof\PP :=(\det\PP)\PP^{-\top}$ together with  the assumption  
$\det\PP=1$. The latter is referred to as an {\it isochoric ansatz}.}{$F_{\rm el}=F\PP^{-1}$.}  \REPLACE{It 
expresses that }{In most materials}, processes  \UUU such \EEE as plastification 
or creep lead 
dominantly to 
shear deformation but not to volumetric variation.  As such, the
isochoric constraint $\det \PP \sim 1$ is often considered. \UUU 
This is taken into account by this model, where values of $\det\PP$ to
be close to $1$ are energetically favored by a specific hardening-like
term, \EEE 
cf.\ Remark~\ref{ref-isochoric} below. \UUU 
Such hardening term (denoted by $\FH$ in \eqref{free-energy+} below)
may control \EEE the full plastic 
strain $\PP$, which  \UUU 
would correspond to the
case of coventional kinematic \EEE hardening 
\UUU in the shear plastic strain \EEE
(which is a relevant model typically 
in metals). In any case, \REPLACE{$\FH$}{the stored energy} will control 
$\det\PP$ to be positive, which 
is important to guarantee \UUU the \EEE (uniform) invertibility of $\PP$ needed in the 
mentioned expression $F_{\rm el}=F\PP^{-1}$. 

The mechanical stored energy $\PsiM$ will have 
elastic and hardening parts $\FE$ and $\FH$ and  will be augmented by 
gradient terms 
and a thermal contribution $\FT$ (considered for simplicity to depend solely on
temperature, i.e., in particular thermal expansion is here neglected), cf.\ 
Remark~\ref{rem-cv-general} for some extension. By integrating on the 
reference configuration $\Omega$, the {\it free energy} 
of the body is expressed by   
\begin{align}\nonumber
&\Psi(\nabla y,\PP,\theta)
=\PsiM(\nabla y,\PP)+\PsiT(\theta)
\ \ \ \text{ with }\ \PsiT(\theta)=\int_\Omega\!\FT(\theta)\,\d x
\\
&\qquad\quad\ \text{ and }\ \ \PsiM(\nabla y,\PP)
=\int_\Omega\FE(\nabla y\,\PP^{-1}
)
+\FH(\PP)
+\frac12\kappa_0\big|\nabla^2y\big|^2
+\frac1q\kappa_1|\nabla \PP|^q\, \d x. 
\label{free-energy+}
\end{align} 
In particular,
the $\kappa_0$-term qualifies the material as {\it 2nd-grade
  nonsimple}, also called {\it multipolar} or {\it complex}, see the
seminal  \cite{Toup62EMCS}  and
\cite{FriGur06TBBC,MinEsh68FSTL,Podi02CISM,PoGiVi10HHCS,Silh88PTNB,TriAif86GALD}. 
On the other hand, the $\kappa_1$-term describes nonlocal plastic
effects and is inspired to the by-now classical {\it gradient
  plasticity} theory \cite{Fleck1,Fleck2,Aifantis91}. In particular,
its occurrence turns out to be crucial in order to prevent the formation
of plastic microstructures and ultimately ensures the necessary
compactness for the analysis. Note that in the finite-plasticity context, the introduction of suitable regularizing
terms on the plastic variables
seems at the moment unavoidable
\cite{Mie-Mull06,MaiMie09GERI,cplas_part2}. 
The exponent $q$ in the $\kappa_1$-term  is given and fixed to be
larger than $d$,  which eases  some points of the analysis. Note however that the
choice $\kappa_1(\nabla \PP) |\nabla \PP|^2$ for some $ \kappa_1 (\nabla
\PP)\sim 1+ |\nabla \PP|^{q-2}$ could be considered as well. 

The {\it frame-indifference} of the  mechanical stored  energy \eqref{frame-indif} 
translates in terms of $\Psi$ as
\begin{align*}
\forall Q\in{\rm SO}(d):\qquad\Psi(Q\nabla y,\PP,\theta)
=\Psi(\nabla y,\PP,\theta).
\end{align*}
In particular let us note that the gradient terms 
are  frame-indifferent  as well.

%
 
The partial functional derivatives of $\Psi$
 give origin to  corresponding 
driving forces.  We use the symbol ``$\,\partial_w\,$'' to indicate both
partial differentiation with respect to the variable $w$ of a smooth 
functional 
or subdifferentiation of a convex 
functional. In case of a single-argument smooth functional, we
write shortly $(\cdot)'$. 

The {\it second  Piola-Kirchhoff 
stress} $\varSigma_{\rm el}$,  here  augmented by a contribution 
arising from the  gradient $\kappa_0$-term, is defined   as 
\begin{subequations}\label{stresses}\begin{align}
&\varSigma_{\rm el}=
  \partial_{\nabla y}\Psi  =
 \FE'(\nabla y\PP^{-1})\PP^{-\top} 
\label{elastic-stresses}
-\kappa_0{\rm div}\nabla^2y.
\ \ \ \ 
\intertext{Furthermore, the
{\it driving stress} for the plastification, again involving a contribution 
arising from the gradient $\kappa_1$-term, reads}
&\label{def-of-Sin}
\varSigma_{\rm in}=\partial_{\PP}^{}\Psi=
\nabla y^\top\FE'(\nabla y\PP^{-1}){:}(\PP^{-1})'+\FH'(\PP)    
 -{\rm div}(\kappa_1|\nabla\PP|^{q-2}\nabla\PP)
.
\end{align}\end{subequations}
Here  and in what follows, we use the (standard) notation ``$\,\cdot\,$'' and ``$\,:\,$'' 
and ``$\,\Vdots\,$''  for the  contraction 
product of vectors, 2nd-order, and 3rd tensors, respectively. 
 The term  $(\PP^{-1})'$ is a 4th-order tensor, namely
$$
(\PP^{-1})'=\Big(\frac{{\Cof}'\PP^\top}{\det\PP}-\frac{\Cof\PP^\top\otimes\Cof\PP}{(\det\PP)^2}\Big)
$$
where $\Cof \PP = (\det \PP) \PP^{-\top}$ is the classical {\it
  cofactor matrix}. The 
product 
$\nabla y^\top\FE'(\nabla y\PP^{-1}){:}(\PP^{-1})'$ turns out to be 
a 2nd-order tensor, as expected. 
Both variations of $\Psi$ above are taken with
respect to the corresponding $L^2$ topologies. 

Eventually, the 
{\it entropy} $\eta$, the {\it heat capacity} $c_{\rm v}$, and the
{\it thermal part}
$\vartheta$  of the internal energy are  classically recovered as 
\begin{align}
&\eta=-\psi_\theta'=-\FT'(\theta),\ \ \ \ 
c_{\rm v}=-\theta\psi_{\theta\theta}''=-\theta\FT''(\theta),\ \ 
\text{ and }\ \ \vartheta=\FT(\theta)-\theta\FT'(\theta).
\label{p-eta-cv}
\end{align}
Note  in particular that 
$\DT\vartheta=\FT'(\theta)\DT\theta-\DT\theta\FT'(\theta)-\theta\FT''(\theta)\DT\theta=c_{\rm v}(\theta)\DT\theta$. The {\it entropy equation} reads 
\begin{align}\label{ent-eq}
\theta\DT\eta+{\rm div}\,j=\text{dissipation rate}.
\end{align}
 We assume 
the heat flux $j$  to be  governed by the {\it Fourier} law
$j=-\mathscr{K}\nabla\theta$  where is  $\mathscr{K}$ the  {\it
  effective heat-transport tensor}, see \eqref{M-pull-back+} below. 
Substituting $\eta$ from \eqref{p-eta-cv} into \eqref{ent-eq}, we 
obtain 
the {\it heat-transfer equation}
\begin{align*}
c_{\rm v}(\theta)\DT\theta-{\rm div}(\mathscr{K}\nabla\theta)=\text{dissipation rate}.
\end{align*}
Note that,  $c_{\rm v}$ depends \UUU on the \EEE temperature only  through 
$\FT$. 

The  model consists  \UUU of \EEE the following  system of
semilinear equations  
\begin{subequations}\label{system}
\begin{align}\label{momentum-eq}
&\varrho\DDT y={\rm div}
\,
\varSigma_{\rm el}
+
g(y),\!\!&&\text{({\sf momentum equilibrium})}
\\&\label{flow-rule-pi}
\partial_\PR^{}\mathfrak{R}\big(\theta;\DT\PP\PP^{-1}\big)
+
\,\varSigma_{\rm in}\PP^\top
\ni0,
&&\text{({\sf flow rule for \UUU the \EEE inelastic strain})}
\\
&c_{\rm v}(\theta)\DT\theta
={\rm div}\big(\mathscr{K}(\PP,\theta)\nabla\theta\big)+
\partial_\PR^{}\mathfrak{R}\big(\theta;\DT\PP\PP^{-1}){:}(\DT\PP\PP^{-1})
\!\!
&&\text{({\sf heat-transfer equation})}
\label{system-heat}
\end{align}
\end{subequations}
where $\mathfrak{R}=\mathfrak{R}\big(\theta;\PR)$ is the (possibly nonsmooth)
{\it (pseudo)potential} related to dissipative forces of viscoplastic
origin, $\PR$ is 
the placeholder for the plastification rate $\DT\PP\PP^{-1}$, and 
$\partial_\PR^{}\mathfrak{R}$ stands for the \DELETE{convex} subdifferential of 
the function $\mathfrak{R}(\theta;\cdot)$, which is assumed  to be
 convex. The right-hand side of \eqref{momentum-eq} features
the pull-back $g{\circ}y$ of the {\it actual} gravity force
$g:\R^d\to\R^d$. This is the most simple example of a  {\it
  nondead}  load, \UUU i.e., a load given in terms of the actual
deformed configuration of the body. \EEE

The {\it effective  heat-\UUU transfer \EEE tensor} 
$\mathscr{K}$ is to be related with the symmetric {\it heat-conductivity tensor}
 $\Frakk=\Frakk( \theta )$
which is  a  given material property,  see \eqref{ass-M-K}
below.  
 The need for such {effective} quantities stems  from  the fact that 
driving forces are to be considered Eulerian in nature, so that a
pull-back to the reference configuration is imperative. \UUU We use
the adjective {\it effective} to indicate quantities that are
defined on the reference configuration but act on the actual one. A first choice
in this direction is  \EEE
\begin{align}
\mathscr{K}(F,\theta)
&=(\det F)F^{-1}\Frakk(\theta) F^{-\top}\!
=(\Cof F^\top)\Frakk(\theta)F^{-\top}\!
\label{M-pull-back}
=\frac{(\Cof F^\top)\Frakk(\theta)\Cof F}{\det F}.
\end{align}
This is just the usual pull-back transformation of 2nd-order covariant tensors. 
In the isotropic case $\Frakk(\theta)=k(\theta)\mathbb{I}$, 
relation  
\eqref{M-pull-back} can also be written 
by using the right Cauchy-Green tensor $C$ as 
\begin{align}\label{M-pull-back-iso}
\mathscr{K}(C,\theta)  =\det C^{1/2} k(\theta)C^{-1}
\quad\text{with }\ \ C=F^\top F, 
\end{align}
cf.\ \cite[Formula (67)]{DuSoFi10TSMF} or
\cite[Formula (3.19)]{GovSim93CSD2} for mass transport instead of
heat. \UUU In fact, the effective transport-coefficient tensor 
is a function of $C$ in general anisotropic cases as well,  cf.\ 
\cite[Sect.\,9.1]{KruRou18MMCM}. In view of this, \EEE
  we now use our smallness assumption   $\UUU E_{\rm el}\sim 0$,
which yields only $F^\top F\sim\PP^\top\PP$, \EEE
in order to infer that \UUU we can, in fact, substitute \EEE $F$ 
with  $\PP$
into \eqref{M-pull-back} as a good modelling ansatz\UUU, even though 
there need not be $\PP\sim F$\EEE. 
Thus, 
the relation \eqref{M-pull-back} 
\UUU turns into \EEE
\begin{align}\label{M-pull-back+}
\mathscr{K}(\PP,\theta)=
\frac{(\Cof\PP^\top)\Frakk(\theta)\Cof\PP}{\det\PP}.
\end{align}
This expression bears the advantage of being independent of $(\nabla y)^{-1}$, 
which turns out useful in relation with estimation and passage to the limit
arguments,  cf.\ \cite{KruRou18MMCM,RouTom??TMPE}.





The plastic flow rule \eqref{flow-rule-pi}  complies with the
so-called {\it plastic-indifference} requirement. Indeed,  the
evolution is 
insensible to prior plastic deformations, for the stored energy 
\begin{equation*}
\widehatFM=\widehatFM(F,\PP):=\FE(F_{\rm el})\qquad\text{ with }\ \
F_{\rm el}=F\PP^{-1}
\end{equation*}
and the dissipation potential  
$$
\widehat{\mathfrak{R}}=
\widehat{\mathfrak{R}}(\PP,\theta;\DT\PP)=\mathfrak{R} (\theta;\DT \PP \PP^{-1})
$$
\UUU complies with the following invariant properties \EEE
\begin{align}
&\widehatFM(FQ,\PP Q)=\widehatFM(F,\PP),\ \ \ \ \FH(\PP Q)=\FH(\PP),
\ \text{ and }\ 
\widehat{\mathfrak{R}}(\PP Q,\theta;\DT\PP Q) 
=\widehat{\mathfrak{R}}(\PP,\theta;\DT\PP)
\label{plast-indif}\end{align} 
for  any  $Q\in{\rm SO}(d)$,  as  discussed  by
Mandel \cite{Mand72PCV} and later  in \cite{Miel02FELG,MiRoSa??GERV}. 
 In particular,  we can equivalently test the flow rule
\eqref{flow-rule-pi} by $\DT\PP\PP^{-1}$ or  rewrite it  as  
\begin{align}\label{flow-rule-pi+}
\partial_\PR^{}\mathfrak{R}\big(\theta;\DT\PP\PP^{-1}\big)\PP^{-\top}
+\,\varSigma_{\rm in}
= \partial_{\DT\PP}\widehat{\mathfrak{R}}\big(\PP,\theta;\DT\PP\big)+\,\varSigma_{\rm in}\ni0
\end{align}
and test it on by $\DT\PP$ obtaining (at least formally) that
\begin{align}\label{test-of-flow-rule}
\partial_\PR^{}\mathfrak{R}\big(\theta;\DT\PP\PP^{-1}\big)
{:}\DT\PP\PP^{-1}=-\varSigma_{\rm in}\PP^\top\!{:}\DT\PP\PP^{-1}
=-\varSigma_{\rm in}\PP^\top\PP^{-\top}\!{:}\DT\PP=-\varSigma_{\rm in}{:}\DT\PP,
\end{align}
where we used also the algebra $AB{:}C=A{:}CB^\top$.

%
 %

System \eqref{system} has to be complemented by 
 suitable boundary and initial conditions.  As for the former
we prescribe 
\begin{subequations}\label{BC}
\begin{align}\label{BC-1}
&\varSigma_{\rm el}\nu-
\divS\big(\kappa_0\nabla^2y
\big)
+Ny=Ny_\flat(t),\ \ \ \ \ \ \ \ \ \kappa_0\nabla^2y{:}(\nu\otimes\nu)=0,
\\&
\PP =\mathbb{I},\ \ \ \ \ \ \text{ and }\ \ \ \ \ \ \mathscr K(\PP,\theta)\nabla\theta{\cdot}\nu+K\theta=K\theta_\flat(t)\qquad\text{on }\ \partial\Omega.
\label{BC-2}
\end{align}\end{subequations}
Relations \eqref{BC-1} correspond to a Robin-type mechanical condition. In 
particular,  $\nu$ is the external normal at $\partial \Omega$, 
${\rm div}_{_{\rm S}}$ denotes the surface divergence defined as a
trace of the surface gradient (which is a projection of the gradient on the 
tangent space through the projector $\bbI-\nu\otimes\nu$),
and $N$ is the elastic modulus of idealized {\it boundary}
springs. \UUU In particular, $y_\flat$ is the (possibly time-dependent)
position of the elastic support. \EEE
Similarly we prescribe  in \eqref{BC-2}  Robin-type boundary condition 
for temperature,  where $K$ is the boundary 
heat-transfer coefficient and $\theta_\flat$  is   the external temperature. 
We assume $\PP$ to be
    the identity at $\partial \Omega$, meaning that plasticization
    occurs in the bulk only. This is chosen here for the sake of
    simplicity and could  be  weakened by imposing $\PP=\mathbb{I}$ on a
    portion of $\partial \Omega$ or even by a Neumann condition, this
    last requiring however a  more  delicate estimation argument, see
    Remark \ref{lastremark}.
Eventually, initial conditions read 
\begin{align}\label{IC}
&y(0)=y_0,\ \ \ \ \ \ \DT y(0)=v_0,\ \ \ \ \ \ \PP(0)=\PP_0,
\ \ \ \ \ \ \theta(0)=\theta_0.
\end{align}

  The full
model  \eqref{system}  with  \eqref{BC}-\eqref{IC} is thermodynamically
consistent.  This can be checked  
by testing  relations (\ref{system}a) and (\ref{system}b)
 by $\DT y$ and $\DT\PP\PP^{-1}$,  respectively.   By adding up these contributions and using
\eqref{test-of-flow-rule} we  obtain 
the {\it mechanical energy balance} 
\begin{align}\nonumber\!\!
&\frac{\d}{\d t}\bigg(\int_\Omega\frac\varrho2|\DT y|^2
\,\d x
+\PsiM(\nabla y,\PP)
+\int_\Gamma \frac12N|y|^2\,\d S\bigg)
+\int_\Omega \partial_\PR^{}\mathfrak{R}\big(\theta;\DT\PP\PP^{-1}){:}(\DT\PP\PP^{-1})
\,\d x \nonumber
\\
&\quad =\int_\Omega 
g(y){\cdot}\DT y\,\d x
+\int_\Gamma 
Ny_\flat{\cdot}\DT y\,\d S.
\label{energy-conserv}\end{align}
 Note that this computation is formal and that such mechanical  energy balance  can be
rigorously justified in case of smooth solutions only,  which may
not exists as  $y$ lacks time regularity due
to the possible occurrence of shock-waves in the nonlinear hyperbolic system
(\ref{system}a).  
Also the power of the external mechanical load in \eqref{energy-conserv}, 
i.e.\ $y_\flat
{\cdot}\DT y$, is not well defined if $\nabla\DT y$ 
is not controlled.   This term will be hence treated 
  by by-part integration in time later on, see 
\eqref{by-part-boundary}.

By adding to \eqref{energy-conserv} the space integral of the heat equation 
\eqref{system-heat} we obtain the {\it total energy balance}
\begin{align}\nonumber
&\hspace{-1em}\frac{\d}{\d t}\bigg(
\!\!\!\linesunder{\int_\Omega\frac\varrho2|\DT y|^2
+\vartheta\,\d x}{kinetic and heat}{energies in the bulk}\!\!\!
+\!\!\!\!\linesunder{\PsiM(\nabla y,\PP)_{_{_{_{_{_{_{_{_{}}}}}}}}}\!}{mechanical energy}{in the bulk}\!\!\!\!
+\!\!\!\!\!\linesunder{\int_\Gamma \frac12N|y|^2\d S_{_{_{_{_{_{_{}}}}}}}\!}{mechanical energy}{on the boundary}\!\!\!\!\!\!\bigg)
\\[-.2em]
&\hspace{1em}=\!\!\!\!\!\linesunder{\int_\Omega 
g(y)\cdot\DT y\,\d x_{_{_{_{_{_{_{}}}}}}}\!}{power of}{bulk load}\!\!\!\!\!
+\!\!\!\!\linesunder{
\int_\Gamma Ny_\flat\cdot\DT y\,\d S}{power of surface}{load on $\Gamma$}\!\!\!\!
+\!\!\!\!\linesunder{\int_\Gamma K(\theta{-}\theta_\flat)\,\d S}{heat flux}{thru $\Gamma$}\!\!\!.\!\!\!\!
\!
\label{energy-conserv+}\end{align}

From \eqref{ent-eq} with the heat flux $j=-\mathscr{K}\nabla\theta$ and with 
the dissipation rate ( which here equals the  heat production rate)
$r:=\partial_\PR^{}\mathfrak{R}\big(\theta;\DT\PP\PP^{-1}){:}(\DT\PP\PP^{-1})$,
one can read the {\it entropy imbalance}
\begin{align}\label{ent-imbalance}
\frac{\d}{\d t}\int_\Omega\!\eta\,\d x
=\int_\Omega\!\frac{r+{\rm div}(\mathscr{K}\nabla\theta)}{\theta}\,\d x
=\int_\Omega\,\frac{r}{\theta}-\mathscr{K}\nabla\theta{\cdot}\nabla\frac1\theta\,\d x
+\int_\Gamma\frac{\mathscr{K}\nabla\theta}\theta{\cdot}\nu\,\d S
\\\qquad\quad=\int_\Omega\!\!\!\!\linesunder{\;\frac{r}{\theta}+\frac{\mathscr{K}\nabla\theta{\cdot}\nabla\theta}{\theta^2}}{entropy production}{rate in the bulk $\Omega$}\!\!\!\!\!\d x\ 
+\int_\Gamma\!\!\!\!\!\!\!\!\!\!\!\!\linesunder{\;\frac{K(\theta_\flat{-}\theta)}\theta}{entropy flux through}{the boundary $\Gamma$}\!\!\!\!\!\!\!\!\!\!\d S
\ge\int_\Gamma K\Big(\frac{\theta_\flat}\theta-1\Big)\theta\,\d S,
\nonumber
\end{align}
provided $\theta>0$ and $\mathscr{K}$ is positive semidefinite. In particular,
if the system is thermally isolated, i.e.\ $K=0$,
\eqref{ent-imbalance}  states that 
the overall entropy is nondecreasing in time.
This shows consistency with the 2nd Law of Thermodynamics.

Eventually, the 3rd Law of Thermodynamics (i.e.\ non-negativity of
temperature),  holds  as  soon as   the initial/boundary conditions
are  suitably  qualified so that $r \geq 0$.  In fact,   we do not consider any 
adiabatic-type effects, which might cause cooling.  

 We conclude the presentation of the model with a number of
remarks and comments on  \UUU our \EEE modeling choices and  possible extensions.

\begin{remark}[{\sl Kelvin-Voigt viscosity}]\label{rem-viscosity}
\upshape
Viscous mechanical dynamics, i.e., Kelvin-Voigt-type viscosity,
could 
be considered as well.  This gives 
rise to 
a  viscous contribution 
$\sigma_{\rm vi}(F,\DT F)$ to the  second Piola-Kirchhoff stress of the form $FS(U,\DT U)$  where  $S$ a symmetric tensor
and $U^2=F^\top F$, i.e.~$F=QU$ with $Q\in{\rm SO}(d)$, cf.\ \cite{Antm98PUVS}.
 Equivalently, one can express such viscous contribution as 
$\sigma_{\rm vi}(F,\DT F)=F\hat S(C,\DT C)$  for some given
function $\hat S$. \UUU For $\sigma_{\rm vi}$ to have a potential, \EEE  namely  $\sigma_{\rm vi}(F,\DT F)=\partial_{\DT
  F}\zeta(F,\DT F)$, 
frame indifference imposes that 
$\zeta(F,\DT{F})=
\zeta(QF,Q(F{+}A\DT{F}))$  for all $Q\in\mathrm{SO}(d)$ and all
$A\in\R^{d\times d}$ antisymmetric. 
This  in turn forces $\sigma_{\rm vi}$ to be strongly nonlinear. In
case of nonsimple materials, such nonlinear viscosity is
mathematically tractable, although its analysis is very delicate
\cite{MieRou??TKVR}.
In combination with the Maxwellian rheology (=creep), this combination is
sometimes referred as the {\it  Jeffreys  rheology}. 
\end{remark}

\begin{remark}[{\sl Scaling of  the plastic gradient}]\label{rem-gradients}
\upshape
Let us now inspect the relation between the parameter $\kappa_1$
and the length scale of plasticity in the material. To this aim, we 
consider $d=2$ and resort to  a  stratified situation 
where  $F$ 
and $\PP$ are constant the $x_1$ direction. We consider  a  shear band 
of width $2\ell$ along the plane $x_1=0$. By letting $\Omega$ be a
rectangle 
of unit size, $q=2$, 
and $\kappa_0=0$ for simplicity, and letting   $F_{\rm el}=\bbI$ and
thus $\nabla y=F=\PP$, the simplest profile of $\PP$ compatible with 
\eqref{free-energy+} is continuous and piecewise affine
in the $x_2$-coordinate. Assume a time-dependent shear (caused by boundary 
conditions) with a constant velocity in the $x_2$-direction, namely  
\begin{align}\nonumber
y(x)=\bigg(\!\!\begin{array}{c}y_1(x_1,x_2)\\y_2(x_1,x_2)\end{array}\!\!\bigg)
\quad&\text{with}\quad
y_1(x_1,x_2)=\begin{cases}x_1+t&\text{if }x_2>\ell,\\[-.1em]
\displaystyle{x_1+2t-2t\frac{x_2}{\ell}+t\frac{x_2^2}{\ell^2}}
&\text{if }0\le x_2\le\ell,\\[.7em]
\displaystyle{x_1-2t-2t\frac{x_2}{\ell}- t\frac{x_2^2}{\ell^2}}
&\text{if }0\ge x_2\ge-\ell,\\[-.1em]
x_1-t&\text{if }x_2<-\ell,
\end{cases}\quad
\\&\text{and}\quad
y_2(x_1,x_2)=x_2.\nonumber
\end{align}
The corresponding plastic strain reads
\begin{align}\label{P-scaling}
\PP =  F&=\nabla y=\bigg(\!\!\begin{array}{cc}1\!&\!
{\partial y_1}/{\partial x_2}\\
0\!&\!1\end{array}\!\!\bigg)
\qquad\text{with}\quad
\frac{\partial y_1}{\partial x_2}=\begin{cases}
\displaystyle{2t\frac{x_2{-}\ell}{\ell^2}}\!\!&\text{if }0\le x_2\le\ell,
\\[.0em]
\ \ \ 0&\text{if }
|x_2|>\ell,
\!\!\!\!\!\!\!\!\!\!\\[.0em]
\displaystyle{\UUU - \EEE 2t\frac{x_2{+}\ell}{\ell^2}}\!\!&\text{if }0\ge x_2\ge-\ell.
\end{cases}
\end{align}
Therefore, in view of \eqref{P-scaling} the
plastic-strain gradient
 $\nabla\PP$ has only one nonvanishing entry, namely  
\begin{align*}
\UUU \left|\frac{\partial\PP_{12}}{\partial x_2}\right|\EEE\sim\begin{cases}0&\text{if }|x_2|>\ell,\\[-.1em]
2t/\ell^2&\text{if }|x_2|\le\ell.
\end{cases}
\end{align*}
Correspondingly, a bounded energy contribution from the term
$\kappa_1|\nabla\PP|^2$ reveals the scaling
$$
\ell\sim\kappa_1^{1/3}t^{2/3}
$$
In particular, the occurrence of the plastic-strain gradient in the
model has a hardening effect, for the slip zone widens as $t^{2/3}$
by  accommodating  large plastic slips. One could control such core
size by letting $\kappa_1$ decay in time (which would be
rather 
disputable), or resorting \UUU to \EEE computing the plastic-strain gradient in
the  actual configuration. 
Alternatively, one may consider including the plastic-strain gradient  term into the dissipation 
potential rather than \UUU into \EEE the free energy. All these
options seem to give rise to new nonlinearities in  \eqref{system}
which seriously complicate the analysis. We hence prefer to stay with
the case of a constant and very small $\kappa_1$ (and thus a very
narrow slip zone), so that the validity of this simplified model can be
guaranteed on the relevant time scales. 
\end{remark}

\begin{remark}[{\sl Alternatives in gradient terms}]\label{rem-gradients+}
\upshape
 The gradient terms in the free energy 
\eqref{free-energy+}
are all Lagrangian in nature. This choice is here dictated by the sake
of mathematical simplicity. On the other hand,
for $\nabla y$ and $\PP$ one could resort \UUU to \EEE
computing gradients with respect to \UUU the \EEE actual
configurations. These choices however \UUU give rise to \EEE
additional nonlinearities which seem to prevent the possibility of
developing a complete existence theory. 
It would be also physically relevant to consider the elastic strain gradient 
$\nabla F_{\rm el}$ instead of the total strain
gradient $\nabla F=\nabla^2y$ in the free energy \eqref{free-energy+}. 
This would lead to 
replacing $-\kappa_0\,{\rm div}\nabla^2y$ by  $-{\rm div}\nabla
(\nabla y \, \Cof \PP^\top)$ in \eqref{elastic-stresses}, but \UUU it \EEE
would  also  give rise to extra contributions to (\ref{stresses}b,d). 
Ultimately,
these terms bring to additional mathematical difficulties which seem presently out
of reach.

\end{remark}

%
%

\begin{remark}[{\sl Self-interpenetration}]
\upshape
 In   large-strain theories, self-interpenetration of
 matter is usually  excluded by constraining   possible deformations. A possibility in this direction is to
resort to the classical Ciarlet-Ne\v cas condition \cite{CiaNec87ISCN}
which reads
$$\int_\Omega\det(\nabla y(x))\,\d x\le |y(\Omega)|.$$ 
In general terms, \UUU encompassing such condition \EEE seems however to be out of reach, as it is usually
the case for dynamic inelasticity \UUU with constraints. \EEE \DELETE{Note however that in the frame of
our general modeling ansatz $F_{\rm el} \sim \bbI$ one has that  $\det
\nabla y \sim (\det \PP)=1$.  Under such an approximation, the  Ciarlet-Ne\v
cas condition would be rewritten as $|\Omega| \le |y(\Omega)|$. Note
that the
latter condition passes to weak limits in our setting. Indeed, the
gradient term and the fact that $d=2,3$ ensure deformations are
compact in the space of continuous functions, which suffices to pass
to the limit.  }\COMMENT{AGAIN A LAGRANGE MULTILIER TO THIS CONSTRAINT 
SHOULD OCCUR AND IT IS TROUBLESOME!}
\end{remark}

\begin{remark}[{\sl More general thermal coupling}]\label{rem-cv-general}
\upshape
 A general  coupling between $F$ and $\theta$  seems to
call for viscosity in $F$, which is here not considered. On the other
hand,  one can  consider  making $\FT$ in \eqref{free-energy+} dependent 
on $\PP$  as well,  i.e.\ $\FT=\FT(\PP,\theta)$.
The heat capacity  would \UUU then \EEE depend also  on $\PP$, i.e.\ 
$c_{\rm v}(\PP,\theta)=-\theta\partial_{\theta\theta}^{}\FT(\PP,\theta)$ and 
the heat-transfer equation \eqref{system-heat}  would be augmented
by an {\it adiabatic term}, \UUU leading to \EEE     
$$
c_{\rm v}(\PP,\theta)\DT\theta
={\rm div}\big(\mathscr{K}(\PP,\theta)\nabla\theta\big)+
\partial_\PR^{}\mathfrak{R}\big(\theta;\DT\PP\PP^{-1}){:}(\DT\PP\PP^{-1})
+
\big(\partial_{\PP}^{}\FT(\PP,\theta)-\partial_{\PP}^{}\FT(\PP,0)\big)
\DT\PP.
$$
Due to the last term, the analysis is then  substantially 
more complicated  and rests on a  
  a suitable modification of Lemma~\ref{lem-est+} below,
cf.\ also \cite{KruRou18MMCM,MieRou15RIST,Roub13NPDE} for this technique.
\end{remark}

\begin{remark}[{\sl Isochoric plasticity/creep}]\label{ref-isochoric}\upshape
\UUU Isochoric \EEE plasticity or creep would  correspond to the
nonaffine holonomic constraint  $\det\PP=1$. \DELETE{Condition 
\eqref{plast-indif} would then be weakened, requiring it only for 
 the special linear group ${\rm SL}(d)$ instead of ${\rm GL}^+(d)$. }
The flow rule \eqref{flow-rule-pi+} would involve
 a {\it reaction force} to this constraint, resulting
 \REPLACE{to}{ into}
$$
\partial_\PR^{}\mathfrak{R}\big(\theta;\DT\PP\PP^{-1}\big)\PP^{-\top}
+\,\varSigma_{\rm in}\ni\varLambda\Cof\PP\ \ \ \ \text{ and }\ \ \ \ 
\det\PP=1,
$$
where $\varLambda$ is a $\R^{d\times d}$-valued Lagrange multiplier, note that 
we used also the formula $\Cof\PP=(\det\PP)'$. The mathematical analysis of 
such system seems open. A relevant model, which would  be
{\it approximately isochoric} for small $\delta>0$, can be tackled 
by the 
analysis presented in the following Sections~\ref{sec-anal}--\ref{sec-proof} when considering the hardening term of the type 
$$
\FH(\PP):=\begin{cases}\displaystyle{
\frac{\delta}{\max(1,\det\PP)^r}
+\frac{(\det\PP-1)^2}{2\delta}}\!\!&\text{ if }\ \det\PP>0,\\
\qquad+\infty&\text{ if }\ \det\PP\ge0\,;\end{cases}
$$
note that the minimum of this potential is attained just at the set 
${\rm SL}(d)$ of the isochoric plastic strains, 
and that it complies with \UUU condition
\eqref{ass-plast-large-HD-growth} ahead \EEE for $r\ge qd/(q-d)$
and also with the plastic-indifference condition \eqref{plast-indif}.
\end{remark}

\section{Existence of weak solutions}
\label{sec-anal}

This section introduces the definition of weak solution to the
problem and brings to the statement of the our main existence result,
namely Theorem \ref{thm}. Let us start by fixing some notation. 

We will use the standard notation $C(\cdot)$ for the space of continuous bounded
functions, $L^p$ for Lebesgue  spaces, and $W^{k,p}$ for Sobolev spaces whose 
$k$-th distributional derivatives are in $L^p$. Moreover, we will use the 
abbreviation $H^k=W^{k,2}$ and,  for all $p\geq 1$, we let the conjugate
exponent $p'=p/(p{-}1)$  (with $p'=\infty$ if $p=1$), and \UUU we use
the  notation \EEE
$p^*$ for the Sobolev exponent $p^*=pd/(d{-}p)$ for $p<d$,
$p^*<\infty$ for $p=d$, and $p^*=\infty$ for $p>d$.
Thus, $W^{1,p}(\Omega)\subset L^{p^*}\!(\Omega)$ or 
$L^{{p^*}'}\!(\Omega)\subset W^{1,p}(\Omega)^*$=\,the dual to $W^{1,p}(\Omega)$. 
In the vectorial case, we will write $L^p(\Omega;\R^d)\cong L^p(\Omega)^d$ 
and $W^{1,p}(\Omega;\R^d)\cong W^{1,p}(\Omega)^d$. 


Given the fixed time interval $I=[0,T]$, we denote by $L^p(I;X)$ the 
standard Bochner space of Bochner-measurable mappings $I\to X$, where
$X$ is a Banach space. Moreover,  $W^{k,p}(I;X)$ denotes the Banach space of 
mappings in  $L^p(I;X)$ whose $k$-th distributional derivative in time is 
also in $L^p(I;X)$.

Let us list here the assumptions on the data which
are used in the following: 
\begin{subequations}\label{ass}
\begin{align}
&\FE:\R^{d\times d}\to\R^+\ \text{ continuously differentiable},\ \
\label{ass-FM}
\mbox{$\sup_{F\in\R^{d\times d}}^{}$}
\frac{|\FE'(F)|}{1{+}|F|^{2^*/2-1}}<\infty,
\\[-.2em]\nonumber
&\FH:\R^{d\times d}\to[0,+\infty]\ \text{ continuously differentiable on }\ {\rm GL}^+(d),\\
&\qquad\FH(\PP)\ge\begin{cases}\epsilon
/(\det\PP)^r\!\!\!&\text{if }\ \det\PP>0,
\label{ass-plast-large-HD-growth}
\\[-.2em]\quad+\infty&\text{if }\ \det\PP\le0,\end{cases}\ \ \ \ 
r\ge\frac{qd}{r{-}q},
\ \ \ \ q>d,
\\
\label{ass-M-K}&\varrho>0,\ \ \mathbb K
:\R\to\R^{d\times d}\ \ 
\text{continuous, bounded, and
uniformly positive-definite,}
\\&
 \mathfrak{R}(\theta;\PR) = \mathfrak{R}_1(\theta;\PR) +
\mathfrak{R}_2(\theta;\PR),  \label{ass-R-decomp}\\
&\qquad  \mathfrak{R}_1(\theta;\PR) = \sigma_{_{\rm Y}}\!(\theta)|R| \ \
\text{where }\ \sigma_{_{\rm Y}}\!:\R^+\to\R^+\ \text{ is continuous and bounded},\label{ass-R1}\\
&\qquad  \mathfrak{R}_2 :\R_+\times \R^{d\times d}  \to\R^+\ \text{
continuously differentiable},\label{ass-R2}\\
&\nonumber
\exists a_\mathfrak{R}^{}>0\ 
\forall\theta\in\R,\ \PR,\PR_1,\PR_2\in\R^{d\times d}:\ \ \ 
\\&\qquad 
(\partial_{\PR}\mathfrak{R}(\theta;\PR_1)
-\partial_{\PR}\mathfrak{R}(\theta;\PR_2)){:}(\PR_1{-}\PR_2)
\ge a_\mathfrak{R}^{}|\PR_1{-}\PR_2|^2, \label{ass-R-monotone}
\\\label{ass-R}&\qquad \, a_\mathfrak{R} |\PR|^2\le
\mathfrak{R}(\theta;\PR)\le(1+|\PR|^2)/a_\mathfrak{R},\ \ \ \ 
\mathfrak{R}(\theta;-\PR)=\mathfrak{R}(\theta;\PR),
\ \ \ 
\\\label{ass-cv}&
c_{\rm v}: \R_+ \to \R_+  \ \text{ continuous, bounded, 
with positive infimum,}\ 
\\&y_0\!\in\! H^2
(\Omega)^d,\ \ 
v_0\!\in\! L^2(\Omega)^d,\ \ 
\PP_0\!\in\! W^{1,q}
(\Omega)^{d\times d},\ \ 
\FH(\PP_0)\!\in\!L^1(\Omega),\ \ q>d,
\\&
\theta_0\!\in\! L^1(\Omega),\ \ \theta_0\ge0,
\label{ass-IC}
\\&\label{ass-load}
g\in C(\R^d)^d,
\ \  \theta_\flat\in L^1(\Sigma),\ \ \theta_\flat\ge0,
\end{align}\end{subequations}
{where $q$ in (\ref{ass}b,j) refers to the exponent used 
in \eqref{free-energy+}. The nonnegative function 
$\sigma_{_{\rm Y}}\!=\sigma_{_{\rm Y}}\!(\theta)$ is in the position of a 
temperature-dependent {\it yield stress},  i.e., a threshold  triggering 
plastification.


 Note that the polynomial growth assumption on $\FE'$ from
\eqref{ass-FM} is not particularly restrictive for $d=2$ but requires $\FE'(F)
\lesssim |F|^2$ for $d=3$. Such a restricted growth is however
compatible with the assumption that $F_{\rm el}$ is close to the
identity. Indeed, if $\FE$ were a linearization of a
nonlinear elastic energy density  at $F_{\rm el}=\mathbb I$  one would have
$$\FE(F_{\rm el}) = \frac12 (F_{\rm el}{-}\mathbb I){:} {\rm D}^2 \FE
(\mathbb I){:} (F_{\rm el}{-}\mathbb I)$$
so that the bound in \eqref{ass-FM} trivially holds. Note that the
fourth-order tensor $ {\rm D}^2 \FE
(\mathbb I)$ plays here the role of (a multiple of) the 
elasticity-moduli  tensor in linearized elasticity.  

In  the case when $\mathfrak{R}(\theta;\cdot)$ is nonsmooth
at 0, its subdifferential are indeed set-valued and thus 
\eqref{ass-R} is to be satisfied for any selection from the involved 
subdifferentials. On the other hand, the heat-production rate
$\partial_\PR^{}\mathfrak{R}\big(\theta;\DT\PP\PP^{-1}){:}(\DT\PP\PP^{-1})$ 
in \eqref{system-heat}  remains single-valued as
$\partial_\PR^{}\mathfrak{R}\big(\theta;\cdot)$ is multivalued just in
 $\DT\PP\PP^{-1}=0$. 

Testing \eqref{system} by smooth functions and using Green formula in space
(even twice for \eqref{momentum-eq} together with a surface 
Green formula over $\Gamma$), the boundary conditions \eqref{BC},
by-part integration in time for (\ref{system}a,b), and the definition 
of the convex subdifferential $\partial_\PR^{}\mathfrak{R}\big(\theta;\cdot)$ 
for \eqref{flow-rule-pi}, we arrive at  the following. 

 
\begin{definition}[Weak formulation of 
\eqref{system}  with  \eqref{BC}-\eqref{IC}]
\label{def}
We call 
the triple 
$(y,\PP,\theta)$ with
\begin{subequations}\label{weak-form-}\begin{align}
  &y\in L^\infty(I;H^2(\Omega)^d)\cap H^1(I;L^2(\Omega)^d),
\\[-.2em]\label{weak-form-Pi}
  &\PP\in L^\infty(I;W^{1,q}(\Omega)^{d\times d})\cap
H^1(I;L^2(\Omega)^{d\times d})\ \ \text{ with }\ \ \frac1{\det\PP}\in L^\infty(Q),
\\[-.2em]&
\theta
\in L^1(I;W^{1,1}(\Omega)),\ \ \ \ \theta\ge0\ \text{ a.e.\ on }\ Q,
\intertext{a {weak solution} to the initial-boundary-value problem
\eqref{system}  with \eqref{BC}-\eqref{IC}
if}\label{weak-form-Delta-Pi}
&{\rm div}\big(\kappa_1|\nabla\PP|^{q-2}\nabla\PP\big)\in L^2(Q)^{d\times d}
\end{align}\end{subequations}
and if the following hold:
\begin{itemize}\item[\rm (i)]
The  weak formulation of  the momentum balance \eqref{momentum-eq} with 
\eqref{elastic-stresses} 
\begin{subequations}\label{weak-form}\begin{align}\nonumber
&\int_Q\Big(\FE'(\nabla y\PP^{-1}){:}(\nabla\tilde y\,\PP^{-1})
-\varrho\DT y{\cdot}\DT{\tilde y}
+\kappa_0\nabla^2y{\Vdots}\nabla^2\tilde y
\,\d x\, \d t
\\[-.4em]
&\hspace{3em}
+\int_\Sigma\! N y{\cdot}\tilde y\,\d S\, \d t
=\int_Q\!
g(y){\cdot}\tilde y\,\d x\, \d t+
\int_\Omega\!v_0{\cdot}\tilde y(0)\,\d x+\int_\Sigma\!
N y_\flat{\cdot}\tilde y\,\d S\, \d t
\label{momentum-weak}\end{align}
holds for any $\tilde y$ smooth with $\tilde y(T)=0$.
\item[\rm (ii)] The weak formulation of  the plastic flow rule 
\eqref{flow-rule-pi} 
in the form \eqref{flow-rule-pi+} with \eqref{def-of-Sin}
\begin{align}\nonumber
&\int_Q\Big(\mathfrak{R}\big(\theta;\PR\big)
+\FE'(\nabla y\PP^{-1})\PP^\top{:}(\nabla
y((\PP^{-1})')(\PR{-}\DT\PP\PP^{-1})
\\[-.6em]&\hspace{2em}
-{\rm div}\big(\kappa_1|\nabla\PP|^{q-2}\nabla\PP\big)\PP^\top
{:}(\PR{-}\DT\PP\PP^{-1})\Big)\,\d x\, \d t
\ge\int_Q\mathfrak{R}\big(\theta;\DT\PP\PP^{-1})\,\d x\, \d t
\label{weak-form-P}
\end{align}
holds for any $\PR$ smooth.
\item[\rm (iii)] The weak formulation of the heat equation \eqref{system-heat}
\begin{align}&\nonumber
\int_Q\!\mathscr{K}(\PP,\theta)\nabla\theta{\cdot}
\nabla\tilde\theta-C_{\rm v}(\theta)\DT{\tilde\theta}
-\partial_\PR^{}\mathfrak{R}\big(\theta;\DT\PP\PP^{-1}){:}(\DT\PP\PP^{-1})
\tilde\theta\,\d x\, \d t
\\[-.3em]&\qquad\qquad\qquad\qquad\qquad
+\int_\Sigma\! K\theta\tilde\theta\,\d S\, \d t
=\int_\Sigma\! K\theta_\flat\tilde\theta\,\d S\, \d t
+\int_\Omega\! C_{\rm v}(\theta_0)\tilde\theta(0)\,\d x
\end{align}\end{subequations}
holds for any $\tilde\theta$ smooth with $\tilde\theta(T)=0$
and with $C_{\rm v}(\cdot)$ 
a primitive function to $c_{\rm v}(\cdot)$.
\item[\rm (iv)] The remaining initial conditions $y(0)=y_0$ and $\PP(0)=\PP_0$ 
are satisfied.
\end{itemize}
\end{definition}

Let us note that, due to \eqref{weak-form-Pi}\DELETE{ together with 
\eqref{weak-form-Delta-Pi}}, we have
also $\PP^{-1}=\Cof\PP^\top/\det\PP
\in L^\infty(Q)^{d\times d}$ so that in particular
${\rm div}\big(\kappa_1|\nabla\PP|^{q-2}\nabla\PP\big)\PP^\top
{:}\DT\PP\PP^{-1}\in L^1(Q)$  due to \eqref{weak-form-Delta-Pi}  and 
thus \eqref{weak-form-P} has a good sense.



 Our main analytical result is an existence theorem for weak
solutions.  This is to  be seen as a mathematical consistency property of
the proposed model. It reads as follows. 

\begin{theorem}[Existence of weak solutions]\label{thm}
Let the assumptions \eqref{ass} hold. Then,  there exists  a weak solution 
$(y,\PP,\theta)$ in the sense of Definition~{\rm {\ref{def}}}  with $\PP$ valued in ${\rm GL}^+(d)$. 
Moreover, the energy conservation \eqref{energy-conserv+} holds on the time 
intervals $[0,t]$ for all $t\in I$, i.e. 
\begin{align}\nonumber
& \int_\Omega\frac\varrho2|\DT y(t)|^2
+
C_{\rm v}(\theta(t))\,\d x+\PsiM(\nabla y(t),\PP(t))
+\int_\Gamma \frac12N|y(t)|^2\d S
\\[-.2em]\nonumber
&\qquad =\int_0^t\!\!\int_\Omega g(y)\cdot\DT y\,\d x\, \d t
+\int_0^t\!\!\int_\Gamma Ny_\flat\cdot\DT y+K(\theta{-}\theta_\flat)\,\d S\, \d t
\\[-.2em]&\qquad\qquad\ +
\int_\Omega\frac\varrho2|v_0|^2+C_{\rm v}(\theta_0)\,\d x+\PsiM(\nabla y_0,\PP_0)
+\int_\Gamma \frac12N|y_0|^2\d S\,.
\label{energy-conserv++}\end{align}
\end{theorem}

We will prove this  result  in Propositions~\ref{prop-est}-\ref{prop-conv2}
by a regularization, transformation, and approximation  procedure. 
This  also provides a (conceptual) algorithm that is numerically stable  and 
 converges as  the discretization and the regularization parameters
 tend to $0$.

\section{Galerkin approximation, stability, convergence}\label{sec-proof}
 We devote \UUU this \EEE section to the proof of the existence result, namely
 Theorem~\ref{thm}. As already mentioned,  
we apply a constructive method  \UUU providing \EEE an approximation of
 the problem.  This results from combining a regularization in
 terms of the small parameter $\varepsilon>0$ and a Galerkin
 approximation, described by the small parameter $h>0$ instead.  \UUU
 The regularization is aimed on the one hand at smoothing the
 potential $\mathfrak{R} (\theta,R)$ in a neighborhood of $R=0$ and on
 the other hand at
 making the heat-production rate and boundary and initial
 temperatures bounded, see below. We obtain \EEE the existence of approximated solutions,
 their stability (a-priori estimates), and their convergence to weak 
solutions, at least in terms of subsequences. The general philosophy of 
a-priori  estimation  relies on the fact that temperature plays a role in 
connection with dissipative mechanisms only: adiabatic 
effects are omitted and most estimates on the mechanical part of the system are
independent of temperature and its discretization. The estimates and the 
convergence rely on the independence of the heat capacity  from  mechanical 
variables, cf.\ also Remark~\ref{rem-cv-general} above.


Let us begin by detailing the regularization. First, we smoothen the convex 
(but generally nonsmooth) potential 
$\mathfrak{R}(\theta;\cdot):\R^{d\times d}\to\R^+$  in order to be
able to use the  conventional theory of ordinary-differential equations for 
the Galerkin approximation. To this goal, we  use the splitting  
$\mathfrak{R}(\theta;\PR)=
\mathfrak{R}_1(\theta;\PR)+\mathfrak{R}_2(\theta;\PR)$ from
\eqref{ass-R-decomp}.  
Then, we exploit the general Yosida-regularization construction to the 
nonsmooth part, i.e. $\mathfrak{R}_\varepsilon(\theta;\PR):=\mathfrak{R}_{1,\varepsilon}(\theta;\PR)
+\mathfrak{R}_2(\theta;\PR)$ with 
\begin{align}\label{regularization-R} 
\mathfrak{R}_{1,\varepsilon}(\theta;\PR):=
\min_{\widetilde\PR\in\R^{d\times d}}\left(\mathfrak{R}_1(\theta;\widetilde\PR)
+\frac1{2\varepsilon}|\widetilde\PR{-}\PR|^2\right) 
=
\left\{
  \begin{array}{ll}
    \displaystyle\frac{\sigma_{_{\rm Y}}\!(\theta)}{2\varepsilon} |\PR|^2\quad &\text{if}  \
    |\PR|\leq \varepsilon\\[4mm]
\sigma_{_{\rm Y}}\!(\theta)\left(|\PR| {-} \displaystyle\frac{\varepsilon}{2}\right)\quad &\text{if}  \
    |\PR|> \varepsilon,
  \end{array}
\right.
\end{align}
{where $\sigma_{_{\rm Y}}\!$ is from \eqref{ass-R1}. 
Properties  (\ref{ass}d-f)  of the original potential $\mathfrak{R}$ 
 entail analogous properties  for the regularization,  in
particular 
\begin{subequations}\label{ass-R-eps}\begin{align}
&\label{ass-R-eps-smooth}
\exists \varepsilon_0>0\ \forall\,0<\varepsilon\le\varepsilon_0:
\forall\theta\!\in\!\R,\ \PR_1,\PR_2\!\in\!\R^{d\times d}:\ \ \ 
\\&\qquad\qquad\qquad\label{ass-R-reg-monotone}
(\partial_{\PR}\mathfrak{R}_\varepsilon^{}(\theta;\PR_1)
-\partial_{\PR}\mathfrak{R}_\varepsilon^{}(\theta;\PR_2)){:}(\PR_1{-}\PR_2)
\ge\frac12a_\mathfrak{R}^{}|\PR_1{-}\PR_2|^2,
\\\label{ass-R-reg}&\qquad\qquad\qquad\, \frac12a_\mathfrak{R} |\PR|^2\le
\mathfrak{R}_\varepsilon^{}(\theta;\PR)\le(1+|\PR|^2)/a_\mathfrak{R}
\intertext{where $a_\mathfrak{R}^{}>0$ is from \eqref{ass-R}.  In
  addition, we can prove that}
&\label{ass-R-conv}
\forall\theta\!\in\!\R,\ \PR\!\in\!\R^{d\times d},\ \PR\ne0:\ \ \ 
\lim_{\varepsilon\to0,\ \widetilde\theta\to\theta}
\partial_\PR\mathfrak{R}_\varepsilon^{}(\widetilde\theta;\PR)=\partial_\PR\mathfrak{R}(\theta;\PR).
\intertext{Note that $\mathfrak{R}(\theta;-\PR)=\mathfrak{R}(\theta;\PR)$ assumed in 
\eqref{ass-R} implies $\partial_\PR\mathfrak{R}_\varepsilon^{}(\theta;0)=0$ so that
the limit in \eqref{ass-R-conv} exists even if $\PR=0$ and we have}
&\forall\theta\!\in\!\R:\ \ \ 
\lim_{\varepsilon\to0,\ \widetilde\theta\to\theta}
\partial_\PR\mathfrak{R}_\varepsilon^{}(\widetilde\theta;0)=0\in\partial_\PR\mathfrak{R}(\theta;0).
\end{align}\end{subequations} 

In order to simplify the convergence proof, we apply the so-called 
enthalpy transformation  to the heat equation. This consists in rescaling 
temperature by introducing a new variable 
\begin{align}\label{def-of-vartheta}
\vartheta=C_{\rm v}(\theta) 
\end{align}
 where $C_{\rm v}$ is the primitive of $c_{\rm v}$
vanishing in $0$. 
Note that $\DT\vartheta=c_{\rm v}(\theta)\DT\theta$ and that
$C_{\rm v}$ is increasing so that its inverse  $ C_{\rm v}^{-1}$ exists 
and $\nabla\theta=\nabla C_{\rm v}^{-1}(\vartheta)=\nabla\vartheta/c_{\rm v}(\theta)
=\nabla\vartheta/c_{\rm v}(C_{\rm v}^{-1}(\vartheta))$.  Upon
letting 
\begin{align*}
\mathfrak{K}(\PP,\vartheta):=
\frac{\mathscr{K}(\PP,C_{\rm v}^{-1}(\vartheta))}{{c_{\rm v}(C_{\rm v}^{-1}(\vartheta))}},\ \ \ \ 
\end{align*}
 we rewrite and regularize the system \eqref{system} by 
\begin{subequations}\label{system+}
\begin{align}\label{momentum-eq+}
&\varrho\DDT y={\rm div}
\,
\varSigma_{\rm el}
+
g(y),
\\&\label{flow-rule-pi++}
\partial_\PR^{}\mathfrak{R}_\varepsilon\big(C_{\rm v}^{-1}(\vartheta);
\DT\PP\PP^{-1}\big)\PP^{-\top}
+
\,\varSigma_{\rm in}
=0,
\\
&\DT\vartheta
={\rm div}\big(\mathfrak{K}(\PP,\vartheta)\nabla\vartheta\big)
+\frac{\partial_\PR^{}\mathfrak{R}_\varepsilon\big(\theta;\DT\PP\PP^{-1}){:}(\DT\PP\PP^{-1})}
{1+\varepsilon|\DT\PP^{}\PP^{-1}|^2},
\!\!
\label{system-heat+}
\end{align}
\end{subequations}
where $\varSigma_{\rm el}$ and $\varSigma_{\rm in}$ are again from \eqref{stresses}.
 Note that we used \eqref{flow-rule-pi++} in the form \eqref{flow-rule-pi+},
which allows the test by $\DT\PP$, in contrast to \eqref{flow-rule-pi} which 
is to be tested by the product $\DT\PP\PP^{-1}$ which  is  not
legitimate  at the level of 
the Galerkin discretisation. \UUU Let us note that due to the boundedness/growth assumptions \eqref{ass-R}, 
the dissipation rate has a quadratic growth and the regularization of the
heat-production rate in \eqref{system-heat+} is bounded. We are hence in the \EEE
position of resorting to a $L^2$-theory instead of the $L^1$-theory for the
regularized heat problem.

 The boundary conditions are correspondingly 
modified, 
i.e.\ 
$\mathscr K(\PP,\theta)\nabla\theta{\cdot}\nu+K\theta=K\theta_\flat(t)$
in \eqref{BC-2} and $\theta(0)=\theta_0$ in \eqref{IC} modify respectively 
as
\begin{subequations}\label{BC-IC+}\begin{align}
&\mathfrak{K}(\PP,\vartheta)\nabla  C_{\rm
  v}^{-1}(\vartheta)  {\cdot}\nu
+K C_{\rm v}^{-1}(\vartheta)=K\theta_{\flat\varepsilon}\!\!\!\!&&\text{with }\ \ \ 
\theta_{\flat\varepsilon}:=\frac{\theta_\flat}{1+\varepsilon\theta_\flat},
\\
&\vartheta(0)=\vartheta_{0\varepsilon}:=C_{\rm v}(\theta_{0\varepsilon})&&\text{with }\ \ \ \theta_{0\varepsilon}=\frac{\theta_0}{1+\varepsilon\theta_0}.
\end{align}\end{subequations}
%
%
\UUU In particular, 
$ \theta_{\flat,\varepsilon}$ and
$\theta_{0\varepsilon}$  in \eqref{BC-IC+} are bounded.\EEE 

As announced, we use a Galerkin approximation in space for \eqref{system+}.
(which, in its evolution variant, is sometimes referred to as 
{\it Faedo-Galerkin method}). For possible numerical implementation, one can 
imagine a conformal finite element formulation, with $h>0$ denoting
the {\it mesh size}.  Assume for simplicity  that  the sequence of
nested finite-dimensional subspaces $V_h \subset H^2(\Omega)$ invading
$H^1(\Omega)$  is  given. This will make all Laplacians defined in the 
usual strong sense even on the discrete level, allowing \UUU for \EEE
some simplification in the estimates. 
For simplicity,  we  assume that all initial conditions 
$(y_0,\PP_0,\vartheta_{0\varepsilon})$ belong to  
all finite-dimensional subspaces so that  no additional
approximation of such conditions is needed. 


The outcome of the Galerkin approximation is an 
an initial-value problem for a system of ordinary 
differential equations (ODEs). 
In \eqref{est-w} below, we denote $|\cdot|_{h}$  the  seminorm on 
$L^2(I;H^1(\Omega)^*)$ defined by
\begin{align}\label{seminorm}
|\xi|_{h}^*:=\sup \left\{\int_Q\xi v\,\d x\,\d t \ :\
\|v\|_{L^2(I;H^1(\Omega))}\le1,\ v(t)\in V_{h} \ \ \text{for a.e.} \
t \in I\right\}. 
\end{align}
Similar seminorms  (with the same notation) are defined on 
spaces tensor-valued functions. On $L^2$-spaces we let  
\begin{align}\label{seminorm+}
|\xi|_{h}:=\sup \left\{\int_Q\xi{:}v\,\d x\,\d t \ :\
\|v\|_{L^2(Q)^{d\times d}}\le1,\ v(t)\in V_{h}^{d\times d} \ \ \text{for a.e.} \
t \in I\right\}, 
\end{align}
to be used for \eqref{est-Delta-Pi} below. This family of seminorms makes 
the linear spaces $L^2(I;H^1(\Omega)^*)$ 
and $L^2(Q)^{d\times d}$ and $L^2(Q)$, 
metrizable locally convex spaces (Fr\'echet spaces). 
\def\eps{\varepsilon}
Henceforth, we use the symbol $C$ to indicate a
positive constant, possibly depending on data but independent \UUU of \EEE
regularization and discretization parameters. Dependences on such
parameters will be indicated in indices. Our stability result
reads as follows. 

\begin{proposition}[Discrete solution and a-priori estimates]
\label{prop-est}
Let  assumptions \eqref{ass} hold and $\eps,h>0$ be fixed. Then, the Galerkin
approximation of \eqref{system+} with the initial/boundary conditions
\eqref{BC}-\eqref{IC} modified by \eqref{BC-IC+}  admits a solution 
on the whole time interval $I=[0,T]$.  By denoting such solution
as 
$(y_{\eps h},\PP_{\eps h},\vartheta_{\eps h})$, \DELETE{we may find $h_0 >0$
such that} the following estimates hold     
\begin{subequations}\label{est}\begin{align}
&\big\|y_{\eps h}\big\|_{L^\infty(I;H^2
(\Omega)^d)\,\cap\,W^{1,\infty}(I;L^2(\Omega)^d)}^{}\le C,
\\[-.2em]\label{est-P}
&\big\|\PP_{\eps h}\big\|_{L^\infty(I;W^{1,q}(\Omega)^{d\times d})\,\cap\,H^1(I;L^
2(\Omega)^{d\times d})}^{}
\le C\ \ \ \text{ and }\ \ \ 
\Big\|\frac1{\det\PP_{\eps h}}\Big\|_{L^\infty(Q)}\le C,
\\[-.2em]&\label{est-theta-H1}
\big\|\vartheta_{\eps h}\big\|_{L^2(I;H^1
(\Omega))}\le C_\eps,
\\&\label{est-w}
\big|\DT\vartheta_{\eps h}\big|_{h_0}^*\le C_\eps\ \ \ \ \text{ for }\ h_0\ge h>0,
\\&
\big|{\rm div}(|\nabla\PP_{\eps h}|^{q-2}\nabla\PP_{\eps h})\big|_{h_0}^{}\le C
\ \ \ \ \text{ for }\ h_0\ge h>0,
\label{est-Delta-Pi}
\end{align}\end{subequations}
{ where  $C$ and $C_\eps$  are  some constant
  independent of $h$ and $h_0$,  $C $   being independent also of $\eps$.}
\end{proposition}

\begin{proof}[Sketch of the proof]
The existence of a global solution to the Galerkin approximation follows 
directly by the usual successive-continuation argument applied to the 
underlying system of ODEs. 

Let us now move to a priori estimation. We start by recovering the
mechanical energy balance, see \eqref{energy-conserv}. 
In particular,  we use 
$\DT y_{\eps h}$, $\DT\PP_{\eps h}$, and $\DT\zeta_{\eps h}$ as test functions into
each corresponding equation discretized by the Galerkin method. 
More specifically,  using $\DT y_{\eps h}$  as test in the
Galerkin approximation of  \eqref{momentum-eq+} with 
its boundary condition \eqref{BC-1}, we obtain
\begin{align}\nonumber
&\int_\Omega\!\frac\varrho2|\DT y_{\eps h}(t)|^2
+\frac{\kappa_0}2|\nabla^2y_{\eps h}(t)|^2\,\d x
+
\int_\Gamma\frac12N|y_{\eps h}(t)|^2\,\d S
+
\int_0^t\!\!\int_\Omega\! \partial_{ \nabla y } \widehatFM(\nabla y_{\eps h},\PP_{\eps h})
{:}\nabla\DT y_{\eps h}\,\d x\, \d t
\\
&
=\int_0^t\!\!\int_\Omega\!g(y){\cdot}\DT y_{\eps h}\,\d x\, \d t
+\int_0^t\!\!\int_\Gamma Ny_\flat{\cdot}\DT y_{\eps h}\,\d S\, \d t
+\int_\Omega\!\frac\varrho2|v_0|^2+\frac{\kappa_0}2|\nabla^2y_0|^2\,\d x
+\int_\Gamma\frac12N|y_0|^2\,\d S.
\label{test-of-moment}\end{align}
 By testing the Galerkin approximation of 
\eqref{flow-rule-pi++}  
by $\DT\PP_{\eps h}$  one gets 
\begin{align}\nonumber
&\int_\Omega\!\FH(\PP_{\eps h}(t))+
\frac{\kappa_1}q|\nabla\PP_{\eps h}(t)|^q\d x
+\int_0^t\!\!\int_\Omega\!\partial_{\PR}{\mathfrak{R}_\varepsilon^{}}\big(\theta_{\eps h};\DT\PP_{\eps h}\PP_{\eps h}^{-1}\big)
{:}\DT\PP_{\eps h}\PP_{\eps h}^{-1}
\\[-.3em]
&\qquad\qquad+
 \partial_\PP\widehatFM(\nabla y_{\eps h},\PP_{\eps h})
{:}\DT\PP_{\eps h}\d x\, \d t
=\int_\Omega\! \psi_{\rm H}(\PP_0) +  \frac{\kappa_1}q|\nabla\PP_0|^q\d x.
\label{test-of-flow-rule+}
\end{align}
Taking the sum of  
\eqref{test-of-moment}-\eqref{test-of-flow-rule+} and using the calculus 
\begin{align}\nonumber
&  \partial_{\nabla y} \widehatFM{:}\nabla\DT y_{\eps h}
+\partial_{\PP}\widehatFM{:}\DT\PP_{\eps h}
=\frac{\partial}{\partial t}\widehatFM(\nabla y_{\eps h},\PP_{\eps h}),
\end{align} 
we obtain the discrete analogue of \eqref{energy-conserv}. 

The boundary term in \eqref{test-of-moment} contains $\DT y$, which is not
well defined on $\Gamma$.  We overcome this obstruction by  
by-part integration  
\begin{align}\label{by-part-boundary}
&\int_0^t\!\int_\Gamma Ny_\flat{\cdot}\DT y_{\eps h}\,\d S\d
t
=
\int_\Gamma Ny_\flat(t){\cdot}y_{\eps h}(t)\,\d S
-\int_0^t\!\int_\Gamma N\DT y_\flat{\cdot}y_{\eps h}\,\d S\, \d t
-\int_\Gamma Ny_\flat(0){\cdot}y_0\,\d S 
\end{align}
so that this boundary term can be estimated by using the assumption 
\eqref{ass-load} on $y_\flat$. 

These estimates allow us to obtain the bounds   (\ref{est}a,b). More
in detail,  the first estimate in \eqref{est-P}  follows from  the 
coercivity \eqref{ass-R} of $\mathfrak{R}_\varepsilon^{}$ so that we have also 
that   $\DT\PP_{\eps h}^{_{}}\PP_{\eps h}^{-1}$  is bounded  in
$L^2(Q)^{d\times d}$.  In particular, we have here used the
boundary condition on the plastic strain \eqref{BC-2}, see also Remark 
\ref{lastremark}.

 Exploiting \eqref{ass-plast-large-HD-growth} we can use the 
Healey-Kr\"omer Theorem  \cite[Thm. 3.1]{HeaKro09IWSS} for the plastic strain 
instead of the deformation gradient. \UUU In particular, \cite[Thm. 3.1]{HeaKro09IWSS}
states that any function $u\in W^{2,p}(\Omega;\R^d)$ with ${\rm det}\,\nabla u>0$
such that $\int_\Omega|{\rm det}\,\nabla u|^{-q}\,\d x\leq K$ is
such that
$\min_{x\in\bar\Omega}{\rm det}\,\nabla u=:\epsilon>0$, provided 
$p>d$ and $q>pd/(p{-}d)$. In fact, this estimate holds uniformly with
respect to $u$, as $\epsilon$ depends on $K$ and data only. In fact,
by inspecting its proof, see also in
\cite{KruRou18MMCM,RouTom??TMPE}, one easily realizes that this result
holds 
for any matrix field, even if it does not come from a gradient of a
vector field. In
particular, one has that any $P\in W^{1,p}(\Omega;\R^{d\times d})$ with ${\rm det}\,P>0$
such that $\int_\Omega|{\rm det}\,P|^{-q}\,\d x\le K$ fulfills 
$\min_{x\in\bar\Omega}{\rm det}\,P=:\epsilon>0$. \EEE
This gives the second estimate in \eqref{est-P}. 
It is important that it is available even on the Galerkin level, so that
in fact the singularity of $\FH$ is not seen during the evolution and
the Lavrentiev phenomenon is excluded. 
Let us point out that, in the frame of our weak thermal coupling the 
assumption \eqref{ass-R}, these estimates  hold  independently of temperature, 
and thus the constants in (\ref{est}a,b) are independent of $\varepsilon$.

Let us now test the Galerkin approximation of the heat equation 
\eqref{system-heat+} by $\vartheta_{\eps h}$. 
This test is allowed at the level of Galerkin approximation, 
although it does not lead to the total energy balance. 
We obtain
\begin{align}
&\frac{\d}{\d t}\frac12\int_\Omega\!\vartheta_{\eps h}^2
\,\d x
+\int_\Omega\!
\mathfrak{K}(\PP_{\eps h},\vartheta_{\eps h})
\nabla\vartheta_{\eps h}{\cdot}\nabla\vartheta_{\eps h}\,\d x
+\int_\Gamma\! K\vartheta_{\eps h}^2\,\d S
=
\int_\Omega\! r_\eps\vartheta_{\eps h}\,\d x+\int_\Gamma\!
K\theta_{\flat\eps}\vartheta_{\eps h}\,\d S.
\label{heat-tested}
\end{align}
After integration over $[0,t]$, we use the Gronwall inequality 
and exploit the control of the initial condition 
$|\theta_{0\eps}|\le1/\eps$ due to the regularization \eqref{BC-IC+}. 
The last boundary term in \eqref{heat-tested}  can be controlled as  
$|\theta_{\flat,\eps}|\le1/\eps$, again due to the regularization \eqref{BC-IC+}. 
Using the positive definiteness of $\mathbb{K}$ in \eqref{ass-M-K} and 
recalling \eqref{M-pull-back+}, we get the bound 
$\|(\Cof\PP_{\eps h})\nabla\theta_{\eps h}/\sqrt{\det\PP_{\eps h}}\|_{L^2(Q)^d}\le
C_\eps$. Then also  \eqref{est-theta-H1} by using
\begin{align}\nonumber
\|\nabla\theta_{\eps h}\|_{L^2(Q)^d}
&=\Big\|\frac{\PP_{\eps h}^{\top}\Cof\PP_{\eps h}}{\det\PP_{\eps h}}\nabla\theta_{\eps h}\Big\|_{L^2(Q)^d}
\\&\le\Big\|\frac{\PP_{\eps h}^{}}{\sqrt{\det\PP_{\eps h}}}\Big\|_{L^\infty(Q)^{d\times d}}
\Big\|\frac{\Cof\PP_{\eps h}}{\sqrt{\det\PP_{\eps h}}}\nabla\theta_{\eps h}\Big\|_{L^2(Q)^d}\le C_\eps,
\end{align}
\UUU where the latter bound follows from \eqref{est-P}. \EEE

By comparison, we obtain the estimate \eqref{est-w} of $\DT\vartheta_{\eps h}$
in the seminorm \eqref{seminorm}.   Again by  comparison we obtain 
\eqref{est-Delta-Pi}, using 
\eqref{flow-rule-pi++} with \eqref{def-of-Sin} and taking advantage of the 
boundedness of the term
$\partial_\PR^{}\mathfrak{R}_\varepsilon^{}(\theta_{\eps h};\DT\PP_{\eps h}\PP_{\eps h}^{-1})\PP_{\eps h}^{-\top}$ in $L^2(Q)^{d\times d}$ and similarly also of 
the first term in \eqref{def-of-Sin}. More specifically,  the  term
\REPLACE{$\nabla y_{\eps h}^\top\FE'(\nabla y_{\eps h}\Cof'\PP_{\eps h}^\top){:}\Cof'\PP_{\eps h}^\top$,}{$\nabla y_{\eps h}^\top\FE'(\nabla y_{\eps h}\PP_{\eps h}^{-1}){:}(\PP_{\eps h}^{-1})'$} turns
out to be bounded in $L^2(Q)^{d\times d}$ because $\nabla y_{\eps h}^\top$ is 
bounded in $L^\infty(I;L^{2^*}(\Omega)^{d\times d})$
and $\FE'$ is bounded in
$L^\infty(I;L^{2^*2/(2^*-2)}(\Omega)^{d\times d})$ due to the growth condition 
\eqref{ass-FM}, 
and \REPLACE{$\Cof'\PP_{\eps h}^\top$}{$(\PP_{\eps h}^{-1})'$} is controlled in 
$L^\infty(Q)^{ d\times d\times d\times d}$. Here, we  emphasize that 
one cannot perform  on relation  \eqref{flow-rule-pi}  the nonlinear test by 
${\rm div}(|\nabla\PP_{\eps h}|^{q-2}\nabla\PP_{\eps h})$ to obtain the
estimate \eqref{est-Delta-Pi} in the full $L^2(Q)^{d\times d}$-norm.
%
\end{proof}

\begin{proposition}[Convergence of the Galerkin approximation for $h\to0$]\label{prop-conv1}
Let  assumptions  \eqref{ass} hold and let $\eps>0$ be fixed. 
Then, for $h\to0$, there  exists a not relabeled  subsequence of 
$\{(y_{\eps h},\PP_{\eps h},\vartheta_{\eps h})\}_{h>0}^{}$
converging weakly* in the topologies indicated in \eqref{est}{\rm a-g}
to some $(y_\eps,\PP_\eps,\vartheta_\eps)$.
Every such limit  triple  is a weak solution 
to the regularized problem \eqref{system+} with the initial/boundary conditions
\eqref{BC}-\eqref{IC} modified by \eqref{BC-IC+}. Moreover, the following
a-priori estimate holds
\begin{align}\label{est++}
&
\big\|{\rm div}(|\nabla\PP_{\eps}|^{q-2}\nabla\PP_{\eps})\big\|_{L^2(Q)^{d
    \times  d}}\le C.
\end{align}
 Furthermore,  the following strong convergences hold for $h\to0$
\begin{subequations}\label{strong-conv}
\begin{align}&&&\label{strong-conv-DT-Pi}
\DT\PP_{\eps h}^{}\PP_{\eps h}^{-1}\to\DT\PP_\eps^{}\PP_\eps^{-1}&&\text{strongly in }\ L^2(Q)^{d\times d},
\\&&&\label{strong-conv-Pi-e-h}
\nabla\PP_{\eps h}\to\nabla\PP_\eps&&\text{strongly in }\ L^q(Q)^{d\times d\times d}.&&&&
\end{align}\end{subequations}
\end{proposition}

\begin{proof}
The existence of weakly* converging  not relabeled 
subsequences  follows by the classical 
 Banach selection principle.
Let us  indicate one such weak* limit by   
$(y_\eps,\PP_\eps,\vartheta_\eps)$  and  prove that it solves the regularized
problem \eqref{system+}.
Note that, estimates \eqref{est++} follow from 
\UUU \eqref{est-Delta-Pi}\EEE, which are independent of $h$ and $h_0$, 
cf.\ \cite[Sect.~8.4]{Roub13NPDE} for this technique. 
\UUU More in detail, one can consider a Hahn-Banach extension of the
linear bounded functional occuring in \eqref{est-Delta-Pi} from
the linear subspace of $L^2(Q)^{d\times d}$ as in the definition \eqref{seminorm+} 
of the seminorm $|\cdot|_h$ to the whole space $L^2(Q)^{d\times d}$. This 
extension is bounded, sharing the same bound $C$ as in \eqref{est-Delta-Pi}.
Selecting, for a moment, another subsequence of these extensions which
converges weakly in $L^2(Q)^{d\times d}$, one can eventually identify the limit
again as ${\rm div}(|\nabla\PP_{\eps}|^{q-2}\nabla\PP_{\eps})\in L^2(Q)^{d\times d}$ and
see that, in fact, the whole originally selected subsequence converges
as well. \EEE

In order to check that weak* limits are solutions, we are called to prove 
convergence of the dissipation-rate term, i.e.\ the heat-production rate, in the
heat-transfer equation.  This in turn requires that we prove the strong 
convergence of $ \PP_{\eps h}$,  i.e.\ (\ref{strong-conv}). To this aim, let
$\tilde\PP_{h}$ be elements of the finite-dimensional subspaces which are 
approximating $\PP_\eps$ \ with respect to strong $L^2$ topologies along with the 
corresponding time derivatives.  Such approximants can be constructed by 
projections at the level of time derivatives. \UUU In particular, one
can ask that
$\tilde \PP_h \to \PP_\eps $ strongly in $H^1(0,T;L^2(\Omega)^{d\times
d}) \cap L^1(0,T;W^{1,q}(\Omega)^{d\times d})$.
\EEE


As for the strong convergence of $\nabla\PP_{\eps h}$, we exploit the uniform
monotonicity of the $q$-Laplacian.  The   Galerkin identity 
related to \eqref{flow-rule-pi++} 
\begin{align}\nonumber
\int_Q 
\nabla y_{\eps h}^\top\FE'(F_{{\rm el},\eps h}){:}
(\PP_{\eps h}^{-1})'{:}\tilde\PP
+
\partial_\PR^{}\mathfrak{R}_\varepsilon^{}(\theta_{\eps h};\DT\PP_{\eps h}\PP_{\eps h}^{-1})
{:}(\tilde\PP\PP_{\eps h}^{-1})\qquad
\\[-.5em]
+
\kappa_1|\nabla\PP_{\eps h}|^{q-2}\nabla\PP_{\eps h}\Vdots\nabla\tilde\PP\,\d x\,\d t
=0
\label{weak-form-P-disc}
\end{align}
 will  be used here for $\tilde\PP:=\PP_{\eps
  h}{-}\tilde\PP_h^{}$  where 
${\tilde\PP}_{h}\to\PP_\eps$  strongly in $H^1(I;L^2(\Omega)^{d\times d})$.
For some constant $c_{d,q}^{}>0$, cf.\  \cite[Lemma I.4.4]{DiBenedetto},  
this allows for  estimating as follows 
\begin{align}\nonumber
&\lim_{h\to0}
c_{d,q}^{}\|\nabla\PP_{\eps h}{-}\nabla\UUU\PP_\eps^{}\EEE\|_{L^q(Q)^{d\times d\times d}}^q
\\&\nonumber\
\le\lim_{h\to0}\int_Q\big(|\nabla\PP_{\eps h}|^{q-2}\nabla\PP_{\eps h}
-|\nabla\PP_{\eps}|^{q-2}\nabla\PP_{\eps}\big)
\Vdots\nabla\big(\PP_{\eps h}-\PP_{\eps}\big)\,\d x\, \d t
\\&\nonumber\ =\lim_{h\to0}\frac1{\kappa_1}\int_Q
\nabla y_{\eps h}^\top\FE'(F_{{\rm el},\eps h}){:}
(\PP_{\eps h}^{-1})'{:}(\PP_{\eps h}{-}\tilde\PP_h^{})
\\[-.5em]&\nonumber\qquad\qquad\qquad
+\partial_\PR^{}\mathfrak{R}_\varepsilon^{}(\theta_{\eps h};\DT\PP_{\eps h}\PP_{\eps h}^{-1})
{:}((\PP_{\eps h}{-}\tilde\PP_h^{})\PP_{\eps h}^{-1})
+\FH'(\PP_{\eps h}){:}(\PP_{\eps h}{-}\tilde\PP_h^{})
\\&\label{strong-conv-grad-Pi}\qquad\qquad\qquad
+|\nabla\PP_{\eps h}|^{q-2}\nabla\PP_{\eps h}\UUU \Vdots \nabla
    (\tilde\PP_h^{}{-}\PP_{\eps}) \EEE
-
|\nabla\PP_{\eps}|^{q-2}\nabla\PP_{\eps}
\Vdots\nabla\big(\PP_{\eps h}-\PP_{\eps}\big)\,\d x\, \d t=0,
\end{align}
where we used $\PP_{\eps h}{-}\tilde\PP_h^{}\to0$ strongly in 
$L^2(Q)^{d\times d}$ due to our estimates \eqref{est-P} \UUU and classical 
Aubin-Lions compact-embedding theorem\EEE.
Moreover,  
$\nabla y_{\eps h}^\top$ is bounded in $L^\infty(I; L^{2^*}\!(\Omega)^{d\times d})$ 
and $\FE'(F_{{\rm el},\eps h})$ is bounded in 
$ L^\infty(I; L^{{2^*2}/(2^*{-}2)}(\Omega)^{d\times d})$  due to the growth 
restriction \eqref{ass-FM}, so that $\nabla y_{\eps h}^\top\FE'(F_{{\rm el},\eps h})$
is bounded in  $L^2(Q)^{d\times d}$.  This allows to pass to the
limit in the 
\UUU term which contains \EEE 
$\nabla y_{\eps h}^\top\FE'(F_{{\rm el},\eps h})$
\UUU and similarly also in the
$\partial_\PR^{}\mathfrak{R}_\varepsilon^{}$-term, by taking into
account that   $\partial_\PR^{}\mathfrak{R}_\varepsilon^{}
(\theta_{\eps h};\DT\PP_{\eps h}\PP_{\eps h}^{-1})$ is bounded in $L^2(Q)^{d\times d}$\EEE.
As for the last term, note that  
$\nabla\PP_{\eps h}\to\nabla\PP_{\eps}$ weakly in $L^q(Q)^{d\times d}$  while
$|\nabla\PP_{\eps}|^{q-2}\nabla\PP_{\eps}\in L^{q'}(Q)^{d\times d}$ is fixed.
Thus \eqref{strong-conv-Pi-e-h} is proved.  From this,
we can also obtain the strong convergence of the $q$-Laplacian
of $\PP_{\eps h}$ in $L^{q'}(I;(W^{1,q}(\Omega)^{d\times d})^*)$, and thus, due to 
the bound \eqref{est-Delta-Pi}, also
\begin{align}\label{q-Laplace-weakly}
{\rm div}(|\nabla\PP_{\eps h}|^{q-2}\nabla\PP_{\eps h})\to 
{\rm div}(|\nabla\PP_{\eps}|^{q-2}\nabla\PP_{\eps})
\ \ \ \text{ weakly in }\ L^2(Q)^{d\times d}.
\end{align}
\DELETE{Here actually we need a bit subtle argument, considering a Hahn-Banach 
extensions on the whole space with the same bound of their norms as in 
\eqref{est-Delta-Pi}.\COMMENT{OK??}}

As for the strong convergence of $\DT\PP_{\eps h}$ 
(or rather of $\DT\PP_{\eps h}^{}\PP_{\eps h}^{-1}$ which occurs in 
the dissipation rate in 
\eqref{system-heat+}), we use  the strong monotonicity \eqref{ass-R-reg} of 
$\partial_\PR^{}\mathfrak{R}_\varepsilon^{}(\theta;\cdot)$ and again
\eqref{weak-form-P-disc} but now 
\UUU with the test function \EEE $\tilde\PP=\DT\PP_{\eps h}-\DT{\tilde\PP}_{h}$. 
%
Taking $a_\mathfrak{R}^{}>0$ from the uniform monotonicity assumption
\eqref{ass-R-monotone}, in view of \eqref{ass-R-reg-monotone},  we can estimate 
\begin{align}\nonumber
&\limsup_{h\to0}
\frac12a_\mathfrak{R}^{}
\big\|\DT\PP_{\eps h}\PP_{\eps h}^{-1}-\DT\PP_{\eps}\PP_{\eps}^{-1}\big\|_{L^2(Q)^{d\times d}}^2
\\[-.0em]&\nonumber\quad\le\limsup_{h\to0}
\int_Q\!\big(
\partial_\PR^{}\mathfrak{R}_\varepsilon^{}(\theta_{\eps h};\DT\PP_{\eps h}\PP_{\eps h}^{-1})-\partial_\PR^{}\mathfrak{R}_\varepsilon^{}(\theta_{\eps h};
\DT\PP_{\eps}\PP_{\eps}^{-1})\big){:}
(\DT\PP_{\eps h}\PP_{\eps h}^{-1}{-}
\DT\PP_{\eps}\PP_{\eps}^{-1})\,\d x\, \d t
\\[-.0em]&\nonumber\quad=\limsup_{h\to0}
\int_Q
\partial_\PR^{}\mathfrak{R}_\varepsilon^{}(\theta_{\eps h};\DT\PP_{\eps h}\PP_{\eps h}^{-1})
{:}(\DT\PP_{\eps h}\PP_{\eps h}^{-1}-
\DT{\tilde\PP}_{h}\PP_{\eps h}^{-1}
)\,\d x\, \d t
\\[-.3em]&\nonumber\qquad\qquad+\lim_{h\to0}
\int_Q
\partial_\PR^{}\mathfrak{R}_\varepsilon^{}(\theta_{\eps h};\DT\PP_{\eps h}\PP_{\eps h}^{-1}){:}
(\DT{\tilde\PP}_{h}\PP_{\eps h}^{-1}-\DT\PP_{\eps}\PP_{\eps}^{-1})
\,\d x\, \d t
\\[-.3em]&\nonumber\qquad\qquad\qquad
-\lim_{h\to0}\int_Q\partial_\PR^{}\mathfrak{R}_\varepsilon^{}(\theta_{\eps h};
\DT\PP_{\eps}\PP_{\eps}^{-1}){:}
(\DT\PP_{\eps h}\PP_{\eps h}^{-1}-
\DT\PP_{\eps}\PP_{\eps}^{-1})\,\d x\d t
\\[-.3em]&\nonumber\quad
=\lim_{h\to0}\int_Q\nabla y_{\eps h}^\top
 \FE'  \big(F_{{\rm el},\eps h}\big){:}
(\PP_{\eps h}^{-1})'{:}(
\DT{\tilde\PP}_{h}-\DT\PP_{\eps h})
+\FH'(\PP_{\eps h}^{}){:}(\DT{\tilde\PP}_{h}-\DT\PP_{\eps h})
\,\d x\d t
\\[-.5em]&\nonumber\qquad\qquad
+\limsup_{h\to0}\int_Q\kappa_1|\nabla\PP_{\eps h}|^{q-2}\nabla\PP_{\eps h}\Vdots
\nabla(
\DT{\tilde\PP}_{h}-\DT\PP_{\eps h})\,\d x\d t
\\[-.3em]&\nonumber\quad
=\lim_{h\to0}\int_Q  \nabla y_{\eps h}^\top
 \FE'\big(F_{{\rm el},\eps h}\big){:}
(\PP_{\eps h}^{-1})'{:}(
\DT{\tilde\PP}_{h}-\DT\PP_{\eps h})
+\FH'(\PP_{\eps h}^{}){:}(\DT{\tilde\PP}_{h}-\DT\PP_{\eps h})
\\[-.5em]&\nonumber\qquad\qquad
+\kappa_1{\rm div}(|\nabla\PP_{\eps h}|^{q-2}\nabla\PP_{\eps h})
{:}
\DT{\tilde\PP}_{h}
\,\d x\d t 
+\limsup_{h\to0}\int_\Omega\frac{\kappa_1}q|\nabla\PP_0|^q-\frac{\kappa_1}q|\nabla\PP_{\eps h}(T)|^q\,\d x
\\[-.3em]&\quad\le
-\int_Q\kappa_1
{\rm div}(|\nabla\PP_{\eps}|^{q-2}\nabla\PP_{\eps}){:}\DT\PP_{\eps}
\,\d x\, \d t 
+\int_\Omega\frac{\kappa_1}q|\nabla\PP_0|^q-\frac{\kappa_1}q|\nabla\PP_{\eps}(T)|^q\,\d x=0,
\label{large-plast-strong-conv}
\end{align}
where we used that $\nabla\DT\PP_{\eps h}$ is well defined  at the
level of Galerkin approximations  (although not in the limit) and
we also used the fact that
\begin{align}\nonumber
&\liminf_{h\to0}\int_Q|\nabla\PP_{\eps h}|^{q-2}\nabla\PP_{\eps h}\Vdots
\nabla(\DT\PP_{\eps h}{-}\DT{\tilde\PP}_{h})\,\d x\, \d t
\\&\nonumber
=\liminf_{h\to0}\int_\Omega\frac1q|\nabla\PP_{\eps h}(T)|^q
-\frac1q|\nabla\PP_0|^q\,\d x
+\lim_{h\to0}\int_Q\!
{\rm div}(|\nabla\PP_{\eps h}|^{q-2}\nabla\PP_{\eps h}){:}\DT{\tilde\PP}_{h}\,\d x\, \d t
\\&\ge\int_\Omega\frac1q|\nabla\PP_{\eps}(T)|^q-\frac1q|\nabla\PP_0|^q\,\d x
+\int_Q\!
{\rm div}(|\nabla\PP_{\eps}|^{q-2}\nabla\PP_{\eps}){:}\DT\PP_{\eps}
\,\d x\, \d t
\label{calculus-Pi}
\end{align}
 as well as  the Green formula combined with the by-part integration in time:\COMMENT{ THE FIRST INTEGRAL CAN BE OMITTED AS I SIMPLIFIED \eqref{calculus-Pi}, OK?}
\begin{align}
\DELETE{\int_Q\big(
{\rm div}(|\nabla\PP_\eps|^{q-2}\nabla\PP_\eps)
\PP_\eps^\top\big){:}(\DT\PP_\eps\PP_\eps^{-1})\,\d x\, \d t
=}\int_Q{\rm div}(|\nabla\PP_\eps|^{q-2}\nabla\PP_\eps)
{:}\DT\PP_\eps\,\d x\, \d t
=\int_\Omega
\frac1q|\nabla\PP_0|^q-\frac1q|\nabla\PP_\eps(T)|^q\,\d x.
\label{calculus-Pi+}
\end{align}
Note that the last term above makes sense as $t\mapsto\nabla\PP_\eps(t)$ is 
weakly continuous to $L^q(\Omega)^{d\times d}$ due to
\eqref{est-P}.  The other integrals in \eqref{calculus-Pi} are
also well-defined due to  estimate \eqref{est-Delta-Pi}. 
Note also that we used \eqref{q-Laplace-weakly} here. 
Moreover, it is possible to show that $-{\rm div}(|\nabla\PP|^{q-2}\nabla\PP)$ 
is indeed the subdifferential of the potential 
$\PP\mapsto\int_\Omega\frac1q|\nabla\PP|^q\,\d x$  in  
$L^2(\Omega)^{d\times d}$. In particular, the chain rule in  
\eqref{calculus-Pi+}  
holds true, see \cite{Brez73OMMS}. This proves convergence   
\eqref{strong-conv-DT-Pi}.

The convergence of the mechanical part for $h\to0$ is  now
straightforward. As   the 
highest-order term in \eqref{momentum-eq+} is linear, weak convergence together 
and  Aubin-Lions compactness for lower-order terms  suffices. The limit 
passage in the quasilinear $q$-Laplacian as well as in the 
$\mathfrak{R}_\varepsilon^{}$-term in \eqref{flow-rule-pi++} follow from the 
already proved strong convergences \eqref{strong-conv}. 

 Eventually,  the limit passage in the semilinear heat-transfer equation 
\eqref{system-heat+}  can be ascertained due 
to the already proved strong convergences  (\ref{strong-conv}a,b),
 allowing indeed the passage to the limit in the (regularized) right-hand side. 
\end{proof}

 In order to remove the regularization by passing to the   limit for
$\eps\to0$, we cannot  directly  rely on the estimates
\eqref{est-theta-H1}-\eqref{est-w}, 
which are dependent on  $\eps>0$.
On the other hand,
having already  passed to the limit in $h$ we are now in the
position of performing a number of nonlinear tests for the heat equation, which are
specifically tailored to the $L^1$-theory. 

\begin{lemma}[Further a-priori estimates for
  temperature]\label{lem-est+}
 Let $\vartheta_\eps$ be the (rescaled) temperature component of
the weak solution to the regularized problem \eqref{system+}, whose
existence is proved  in 
Proposition~{\rm \ref{prop-conv1}}.  Then, 
\begin{equation}
  \label{positivity}
  \vartheta_\eps \geq 0 \quad \text{a.e. in} \ Q.
\end{equation}
 Moreover, one has that 
\begin{subequations}\label{est+}\begin{align}\label{est-theta-L1}
&\exists C_1>0: \quad \|\vartheta_\eps\|_{L^\infty(I;L^1(\Omega))}\le C_1,
\\&
\forall 1\le s<(d{+}2)/(d{+}1) \ \exists C_s >0 : \quad
    \|\nabla\vartheta_\eps\|_{L^s(Q)^d}\le C_s
\label{est-theta}
\end{align}\end{subequations}
where the constants $C_1, \, C_s$ are independent of $\eps$.
\end{lemma}




\begin{proof}
 The nonnegativity \eqref{positivity} is readily obtained by
testing  the heat equation in the regularized 
enthalpy form \eqref{system-heat+}
 \UUU by \EEE  $\min(0,\vartheta_\eps)$,
and using the assumptions that $\theta_\flat\ge0$ and $\vartheta_0\ge0$, cf.\ 
(\ref{ass}k,l).  
Note that the normalization  $C_{\rm v}(0)=0$ for the primitive 
function $C_{\rm v}$ of $c_{\rm v}$ in \eqref{def-of-vartheta} is here used. 
The nonnegativity \eqref{positivity} allows us to read the bound 
\eqref{est-theta-L1} from the test of the heat equation by 1.

We now perform a second {\it nonlinear} test in order to gain an
estimate on $\nabla\vartheta$ independent of $\eps$.
We follow the classical \cite{BocGal89NEPE} in the simplified variant developed 
in \cite{FeiMal06NSET}. This calls for testing the heat equation
\eqref{system-heat+} on  $\chi_\omega(\vartheta_{\eps})$ where the increasing 
concave function $\chi_\omega:[0,+\infty)\to[0,1]$  is defined as 
$\chi_\omega(w):=1-1/(1{+}w)^\omega$ for some  small $\omega>0$ to be chosen 
later.     
 By using the nonnegativity of $C_{\rm v}$ and the fact that 
$0\le\chi_\omega(\vartheta_\eps)\le1$, 
$\chi_\omega'(w)=\omega/(1{+}w)^{1+\omega}$, 
$r_\eps\le r$,
and $\theta_\flat/(1{+}\eps\theta_\flat)\le\theta_\flat$,  we
obtain  the estimate
\begin{align}\nonumber
&\omega\int_Q\frac{a_{\mathbb K}^{}}{(1{+}\vartheta_\eps)^{1+\omega}}
\Big|\frac{\Cof\PP_\eps}{\sqrt{\det\PP_\eps}}\nabla\vartheta_\eps\Big|^2
\,\d x\, \d t
\le
\int_Q\mathscr{K}(\PP_\eps,\vartheta_\eps)\nabla\vartheta_\eps\cdot\nabla\chi_\omega(\vartheta_\eps)\,\d x\, \d t
\\\nonumber&\qquad\qquad\qquad
\le
\int_Q\mathscr{K}(\PP_\eps,\vartheta_\eps)\nabla\vartheta_\eps\cdot\nabla\chi_\omega(\vartheta_\eps)\,\d x\, \d t
+\int_\Sigma K\vartheta_\eps\chi_\omega(\vartheta_\eps)\,\d S\, \d t
\\\nonumber
&\qquad\qquad\qquad\le \int_Q r_\eps\,\d x\, \d t
+\int_\Sigma K\frac{\theta_\flat}{1{+}\eps\theta_\flat}\d S\, \d t
+\int_\Omega C_{\rm v}(\vartheta_{0,\eps})\,\d x
\\
&\qquad\qquad\qquad\le \int_Q r\,\d x\, \d t
+\int_\Sigma K\theta_\flat\d S\, \d t
+\int_\Omega C_{\rm v}(\theta_0)\,\d x,
\label{est-of-cofPi-grad-theta}
\end{align}
where $a_{\mathbb K}^{}>0$  stands for  the positive-definiteness constant of 
${\mathbb K}$, cf.\ \eqref{ass-M-K}. By the H\"older inequality, we
 have that  
\begin{align}\nonumber
&\int_Q\Big|
\frac{\Cof\PP_\eps}{\sqrt{\det\PP_\eps}}\nabla\vartheta_\eps\Big|^s\,\d x\, \d t
=\int_Q(1{+}\vartheta_\eps)^{(1+\omega)s/2}
\frac{|(\Cof\PP_\eps)\nabla\vartheta_\eps|^s}
{(\det\PP_\eps)^{s/2}(1{+}\vartheta_\eps)^{(1+\omega)s/2}}\,\d x\, \d t
\\&\nonumber
\hspace{16mm}\le\bigg(\int_Q
(1{+}\vartheta_\eps)^{(1+\omega)s/(2-s)}\,\d x\, \d t\bigg)^{1-s/2}
\bigg(\int_Q\frac{|(\Cof\PP_\eps)\nabla\vartheta_\eps|^2}{\det\PP_\eps(1{+}\vartheta_\eps)^{1+\omega}}
\,\d x\, \d t\bigg)^{s/2}
\end{align}
 and the  last factor  on the right-hand side  is bounded due to 
\eqref{est-of-cofPi-grad-theta}. We now use the Gagliardo-Nirenberg inequality 
\begin{equation}\|v\|_{L^{(1+\omega)s/(2-s)}({\Omega})}
\le C_{_{\rm GN}}\|v\|_{L^1({\Omega})}^{1-\lambda}
\|v\|_{W^{1,s}({\Omega})}^{\lambda}\label{gn}
\end{equation}
with $\|v\|_{W^{1,s}({\Omega})}:=
\|v\|_{L^1({\Omega})}+\|\nabla v\|_{L^s({\Omega})^d}$. The latter
inequality holds for  all
$\lambda \in (0,1)$ such that 
\begin{equation}
\frac{2{-}s}{(1{+}\omega)s} \geq \lambda \left(\frac1s - \frac1d
\right) +1 - \lambda\label{la}
\end{equation}
and, correspondingly, for some $C_{_{\rm GN}}>0$ depending on
$s,\,\lambda$, and $\Omega$. We shall apply inequality \eqref{gn}
along with the  choices
$v=1+\vartheta_\eps(t,\cdot)$ and   $\lambda = (2{-}s)/(1{+}\omega)$. Note
that $\lambda \in (0,1)$ as $s\in [1,(d{+}2)/(d{+}1))\subset
[1,2)$. Moreover, by letting $\omega >0 $ small enough one has that condition
\eqref{la} holds, again by virtue of $s< (d{+}2)/(d{+}1)$. Hence, inequality 
\eqref{gn} gives 
\begin{align}\nonumber
&
\Bigg(
\int_0^T\!\!\big\|1+\vartheta_\eps(t,\cdot)\big\|^{(1+\omega)s/(2-s)}
_{L^{(1+\omega)s/(2-s)}({\Omega})}\,\d t
\Bigg)^{1-s/2}
\\&\nonumber\
\le
\Bigg(
\int_0^T\!\!
C_{_{\rm GN}}^{(1+\omega)s/(2-s)}C_0^{(1-\lambda)(1+\omega)s/(2-s)}
\Big(C_0{+}
\big\|\nabla\vartheta_\eps(t,\cdot)\big\|_{L^s({\Omega})^d}\Big)^{\lambda(1+\omega)s/(2-s)}
\,\d t\Bigg)^{1-s/2}
\\
  & \
\UUU \leq \EEE C_\delta + \delta 
\int_Q
\big|\nabla\vartheta_\eps
\big|
^s
\,\d x\, \d t 
\label{8-***}
\end{align}
 where $C_0=|\Omega| + C_1$ and $C_1$ is from
\eqref{est-theta-L1}. The constant $C_\delta$ depends on $C_0$,
$C_{\rm GN}$, and the small $\delta$, 
cf.\ e.g.\ \cite[Formula (12.20)]{Roub13NPDE},  and  we used the
fact that $$\lambda (1{+}\omega)s/(2{-}s) (1{-}s/2) = s(1{-}s/2)< s.$$ 
 We combine this estimate with the bound 
\begin{align}\nonumber
\big\|\nabla\vartheta_{\eps}\big\|_{L^s(Q)^d}^{ s}
&=\Big\|
\frac{\PP_{\eps}^{\top}
(\Cof\PP_{\eps})}{{\det\PP_\eps}}\nabla\vartheta_{\eps}\Big\|_{L^s(Q)^d}^{ s}
\\&
\le\Big\|
\frac{\PP_{\eps}}{\sqrt{\det\PP_\eps}}\Big\|_{L^\infty(Q)^{d\times d}}^{ s}
\Big\|
\frac{\Cof\PP_\eps}{\sqrt{\det\PP_\eps}}\nabla\vartheta_{\eps}
\Big\|_{L^s(Q)^{d\times d}}^{ s}\;.
\label{est-heat++}
\end{align}
Note that the last term in the right-hand side is what occurs in the 
left-hand side of \eqref{est-of-cofPi-grad-theta}. By choosing $\delta$ small
enough, we deduce from \eqref{est-of-cofPi-grad-theta}, \eqref{8-***},
and \eqref{est-heat++} that 
\begin{align}
\Big\|
\frac{\Cof\PP_\eps}{\sqrt{\det\PP_\eps}}\nabla\vartheta_{\eps}
\Big\|_{L^s(Q)^{d\times
    d}}\le  \tilde C_s 
\label{est-heat+}
\end{align}
for any $1\le s<(d+2)/(d+1)$  and some positive $\tilde C_s$. 
From  the latter bound we directly deduce estimate
\eqref{est-theta} 
by using again \eqref{est-heat++}.
\end{proof}

\begin{proposition}[Convergence of the regularization for $\eps\to0$]\label{prop-conv2}
 Under assumptions \eqref{ass}, as  $\eps\to0$ there 
exists a    subsequence of
$\{(y_{\eps},\PP_{\eps},\vartheta_{\eps})\}_{\eps>0}^{}$
(not relabelled)
which converges weakly* in the topologies indicated in 
{\rm(\ref{est}a-f)}, \eqref{est++}, and 
\eqref{est+} to some $(y,\PP,\vartheta)$.
Every such a limit triple is a weak solution to the original problem 
in the sense of Definition~{\rm \ref{def}}. Moreover, the following 
strong convergences hold
\begin{subequations}\label{strong-conv+}
\begin{align}&&&\label{strong-conv-DT-Pi+}
\DT\PP_{\eps}^{}\PP_{\eps}^{-1}\to\DT\PP\PP^{-1}&&\text{strongly in }\ L^2(Q)^{d\times d},
\\&&&\label{strong-conv-Pi-e}
\nabla\PP_{\eps}\to\nabla\PP&&\text{strongly in }\ L^q(Q)^{d\times d\times d}.&&&&
\end{align}\end{subequations}
 Eventually,  the regularity \eqref{weak-form-Delta-Pi} and 
the energy conservation \eqref{energy-conserv++} hold.
\end{proposition}

\begin{proof}
Again, by the Banach selection principle, we  can extract a not
relabelled  subsequence  \UUU converging \EEE with respect to the topologies from
the estimates \rm(\ref{est}a,b) inherited for $(y_\eps,\PP_\eps)$, 
\eqref{est++}, and \eqref{est+}, and indicate its limit by  $(y,\PP,\vartheta)$.

The improved, strong convergences \eqref{strong-conv+}  can be obtained by 
arguing as in the proof of \eqref{strong-conv} in
Proposition~\ref{prop-conv1}. 

One has just to modify the argument in \eqref{strong-conv-grad-Pi} as:
\begin{align}\nonumber
&\lim_{\eps\to0}c_{d,q}^{}\|\nabla\PP_{\eps}{-}\nabla\PP\|_{L^q(Q)^{d\times d\times d}}^q
\le\lim_{\eps\to0}\int_Q\big(|\nabla\PP_{\eps}|^{q-2}\nabla\PP_{\eps}
-|\nabla\PP|^{q-2}\nabla\PP\big)
\Vdots\nabla\big(\PP_{\eps}-\PP\big)\,\d x\, \d t
\\&\nonumber\quad =\lim_{\eps\to0}\frac1{\kappa_1}\int_Q
\nabla y_{\eps}^\top\FE'(F_{{\rm el},\eps}){:}
(\PP_{\eps}^{-1})'{:}(\PP_{\eps}{-}\PP)
+\partial_\PR^{}\mathfrak{R}_\varepsilon^{}(\theta_{\eps};\DT\PP_{\eps}\PP_{\eps}^{-1})
{:}((\PP_{\eps}{-}\PP)\PP_{\eps}^{-1})\,\d x\, \d t
\\[-.2em]&\label{strong-conv-grad-Pi+}\qquad\qquad\qquad\qquad\qquad\qquad\qquad
-\int_Q|\nabla\PP|^{q-2}\nabla\PP\big)
\Vdots\nabla\big(\PP_{\eps}-\PP\big)\,\d x\, \d t=0.
\end{align}
Here we used that the sequence 
$\partial_\PR^{}\mathfrak{R}_\varepsilon^{}(\theta_{\eps};\DT\PP_{\eps}\PP_{\eps}^{-1})$ is bounded in $L^2(Q)^{d\times d}$ (without caring about its limit) 
while $(\PP_{\eps}{-}\PP)\PP_{\eps}^{-1}\to0$ strongly in $L^2(Q)^{d\times d}$ (or 
even in $L^\infty(Q)^{d\times d}$, cf.\ the arguments used already for 
\eqref{strong-conv-grad-Pi}).
Thus \eqref{strong-conv-Pi-e} is proved.

\UUU For \eqref{strong-conv-DT-Pi+}, one \EEE 
has just to modify the argument in \eqref{large-plast-strong-conv}, 
for the term $\nabla\DT\PP_{\eps}$ is not well defined. Relying on the fact that 
${\rm div}(\kappa_1|\nabla\PP_{\eps}|^{q-2}\nabla\PP_{\eps})\in L^2(Q)^{d\times d}$, 
we have 
\begin{align*}\nonumber
&\limsup_{\eps\to0}\frac12a_\mathfrak{R}^{}
\big\|\DT\PP_{\eps}\PP_{\eps}^{-1}-\DT\PP\PP^{-1}\big\|_{L^2(Q)^{d\times d}}^2
\\[-.0em]&\nonumber\quad\le\limsup_{\eps\to0}
\int_Q\big(\partial_\PR^{}\mathfrak{R}_\varepsilon^{}(\theta_{\eps};\DT\PP_{\eps}\PP_{\eps}^{-1})
-\partial_\PR^{}\mathfrak{R}_\varepsilon^{}(\theta;\DT\PP\PP^{-1})\big)
{:}\big(\DT\PP_{\eps}\PP_{\eps}^{-1}-\DT\PP\PP^{-1}\big)\,\d x\d t
\\[-.0em]&\nonumber\quad=\limsup_{\eps\to0}
\int_Q
\partial_\PR^{}\mathfrak{R}_\varepsilon^{}(\theta_{\eps};\DT\PP_{\eps}\PP_{\eps}^{-1})
{:}(\DT\PP_{\eps}{-}\DT\PP)\PP_{\eps}^{-1}\,\d x\, \d t
\\[-.3em]&\nonumber\qquad\qquad+\lim_{\eps\to0}
\int_Q
\partial_\PR^{}\mathfrak{R}_\varepsilon^{}(\theta_{\eps};\DT\PP_{\eps}\PP_{\eps}^{-1}){:}
\DT\PP(\PP_{\eps}^{-1}{-}\PP^{-1})
\,\d x\, \d t
\\[-.3em]&\nonumber\qquad\qquad\qquad
-\lim_{\eps\to0}\int_Q\partial_\PR^{}\mathfrak{R}_\varepsilon^{}(\theta_{\eps};
\DT\PP\PP^{-1}){:}(\DT\PP_{\eps}\PP_{\eps}^{-1}-\DT\PP\PP^{-1}))\,\d x\d t
\\[-.3em]&\nonumber\quad
\stackrel{ \rm (a)}{=}\lim_{\eps\to0}\int_Q\nabla y_{\eps}^\top\FE'\big(
\nabla y_{\eps} 
\PP_\eps^{-1}
\big){:}
(\PP_\eps^{-1})'
{:}
(\DT\PP_{\eps}{-}\DT\PP)+ \psi_{\rm H}' (\PP_{\eps}){:}(\DT\PP_{\eps}{-}\DT\PP)\,\d x\d t
\\[-.3em]&\nonumber\qquad\qquad
-\liminf_{\eps\to0}\int_Q{\rm div}\big(\kappa_1|\nabla\PP_{\eps}|^{q-2}\nabla\PP_{\eps}\big){:}(\DT\PP_{\eps}{-}\DT\PP)\,\d x\, \d t 
\\[-.3em]&\nonumber\qquad\qquad\qquad
 -\lim_{\eps\to0}\int_Q\partial_\PR^{}\mathfrak{R}_\varepsilon^{}(\theta_{\eps};
\DT\PP\PP^{-1}){:}(\DT\PP_{\eps}\PP_{\eps}^{-1}-\DT\PP\PP^{-1}))\,\d x\d t
\\[-.3em]&\nonumber\quad
\stackrel{ \rm (b)}{=}\lim_{\eps\to0}\int_Q\nabla y_{\eps}^\top\FE'\big(
\nabla y_{\eps} 
\PP_\eps^{-1}
\big){:}
(\PP_\eps^{-1})'{:}(\DT\PP_{\eps}{-}\DT\PP)
+ \psi_{\rm H}'(\PP_{\eps}){:}(\DT\PP_{\eps}{-}\DT\PP)
\\[-.3em]&\nonumber\qquad\qquad
+{\rm div}(\kappa_1|\nabla\PP_{\eps}|^{q-2}\nabla\PP_{\eps}){:}\DT\PP
\,\d x\d t
+\limsup_{\eps\to0}\int_\Omega\frac{\kappa_1}q|\nabla\PP_0|^q-\frac{\kappa_1}q|\nabla\PP_{\eps}(T)|^q\,\d x
\\[-.3em]&\quad\le
\int_\Omega\frac{\kappa_1}q|\nabla\PP_0|^q
-\frac{\kappa_1}q|\nabla\PP(T)|^q\,\d x-\int_Q
{\rm div}(\kappa_1|\nabla\PP|^{q-2}\nabla\PP){:}\DT\PP
\,\d x\d t\nonumber
=0.
\end{align*}
In addition to the arguments analogous to 
\eqref{calculus-Pi}-\eqref{calculus-Pi+},  for equality (a) we
have used the fact that  
\begin{align}\nonumber
&\bigg|\int_Q
\partial_\PR^{}\mathfrak{R}_\varepsilon^{}(\theta_{\eps};\DT\PP_{\eps}\PP_{\eps}^{-1}){:}
\DT\PP(\PP_{\eps}^{-1}{-}\PP^{-1})
\,\d x\, \d t\bigg|
\\&\qquad
\le\|\partial_\PR^{}\mathfrak{R}_\varepsilon^{}(\theta_{\eps};\DT\PP_{\eps}\PP_{\eps}^{-1})\|_{L^2(Q)^{d\times d}}^{}\|\DT\PP\|_{L^2(Q)^{d\times d}}^{}
\|\PP_{\eps}^{-1}{-}\PP^{-1}\|_{L^\infty(Q)^{d\times d}}^{}\to0.
\label{conv---}\end{align}
 This follows as  $\PP_{\eps}^{-1}\DELETE{=\Cof\PP_{\eps}^{\top}}\to\DELETE{\Cof\PP^{\top}=}\PP^{-1}$ strongly 
in $L^\infty(Q)^{d\times d}$ due to our estimates 
\eqref{est-P} inherited for $\PP_\varepsilon$, as used already for 
\eqref{strong-conv-grad-Pi}.  Moreover, for equality (b) we also used that 
\begin{align}\nonumber
\lim_{\eps\to0}
&\int_Q\partial_\PR^{}\mathfrak{R}_\varepsilon^{}\big(\theta_{\eps};
\DT\PP\PP^{-1}\big){:}(\DT\PP_{\eps}\PP_{\eps}^{-1}-\DT\PP\PP^{-1})\,\d
x\d t  
\\&\nonumber\quad
=\lim_{\eps\to0}\int_Q\partial_\PR^{}\mathfrak{R}_{1,\eps}^{}\big(\theta_{\eps};
\DT\PP\PP^{-1}\big){:}(\DT\PP_{\eps}\PP_{\eps}^{-1}-\DT\PP\PP^{-1})\,\d x\, \d t
\\\label{conv--}&\quad\qquad+
\lim_{\eps\to0}\int_Q\partial_\PR^{}\mathfrak{R}_2^{}\big(\theta_{\eps};
\DT\PP\PP^{-1}\big){:}(\DT\PP_{\eps}\PP_{\eps}^{-1}-\DT\PP\PP^{-1})\,\d x\, \d t
=0.
\end{align} 
The latter follows from  (\ref{ass-R-eps}d,e)
 and the fact that  $\theta_\varepsilon\to\theta$
strongly in $L^1(Q)$ hence 
$\sigma_{_{\rm Y}}\!(\theta_\varepsilon)\to\sigma_{_{\rm Y}}\!(\theta)$ a.e.  Indeed, 
we have that 
$\partial_\PR^{}\mathfrak{R}_{1,\varepsilon}^{}(\theta_{\eps};\DT\PP\PP^{-1})$ converges
a.e.\ on $Q$ either to $\partial_\PR^{}\mathfrak{R}(\theta;\DT\PP\PP^{-1})$
if $\DT\PP\PP^{-1}(t,x)\ne0$ or to $0$ otherwise. 
 As  the sequence 
$ \partial_\PR^{}\mathfrak{R}_{1,\eps}^{}(\theta_{\eps};\DT\PP\PP^{-1})
$
 is  bounded in $L^\infty(Q)^{d\times d}$,  the Vitali
theorem ensures that  it converges strongly in  $L^r(Q)^{d\times
  d}$ for all
$r<\infty$.    On the other hand,  
$\partial_\PR^{}\mathfrak{R}_2^{}(\theta_{\eps};\DT\PP\PP^{-1})\to
\partial_\PR^{}\mathfrak{R}_2^{}(\theta;\DT\PP\PP^{-1})$ strongly in 
$L^2(Q)^{d\times d}$ just by the usual continuity of the underlying 
Nemytski\u\i\ mapping.  Since we have  the weak convergence 
$\DT\PP_{\eps}\PP_{\eps}^{-1}\to\DT\PP\PP^{-1}$ in $L^2(Q)^{d\times d}$,
 convergence  \eqref{conv--} follows.


The passage to the limit 
then follows
similarly as in the proof of Proposition \ref{prop-conv1}. 
A little difference concerns the strong convergence of $\vartheta_\eps$,
which follows again by the Aubin-Lions Theorem but we use here a coarser 
topology than in Proposition~\ref{prop-conv1}. \UUU Namely, the
convergence holds in  \EEE$L^p(Q)$ with 
arbitrary $1\le p<1+2/d$,
related to the estimates \eqref{est+} when interpolated. This change is however 
immaterial with respect to the limit passage in the mechanical
part (\ref{system}a,b). Actually, some arguments are even 
simplified, for we do not need to approximate the limit into the 
finite-dimensional 
subspaces as we  did  in 
\eqref{large-plast-strong-conv}.
%
%
The heat-production rate on the right-hand side
of \eqref{system-heat+} converges now strongly in $L^1(Q)$.

 Eventually, regularity \eqref{weak-form-Delta-Pi} can be obtained from the estimates 
\eqref{est++}, which are uniform in $\eps>0$.
%
%
The energy conservation \eqref{energy-conserv++} follows  directly
 from the 
energy conservation in the mechanical part, as essentially used 
above while checking  the strong convergences  \eqref{strong-conv+}. 
Indeed, one integrates \eqref{energy-conserv+} over $[0,t]$ and
\UUU sums \EEE it to the heat equation tested on the constant 1. Note that this
is amenable as the constant $1$ can be put in  
duality with $\DT\vartheta$, so that the chain-rule applies.  
\end{proof}

\begin{remark}[{\sl Boundary conditions on $\PP$}]
\label{lastremark}
\upshape
We have assumed here the homogeneous Dirichlet condition $\PP={\mathbb I}$ on 
$\Sigma$ for the sake of simplicity. One has however to mention that
 other boundary conditions could be considered. \UUU In particular,
 this could be done \EEE
if the hardening $\FH$ \UUU were \EEE coercive on the whole plastic tensor $\PP$, otherwise only  at the expense of some additional 
intricacies. Indeed, a bound on $\PP_{\eps h}$ could be obtained from that on 
$\DT\PP_{\eps h}^{_{}}\PP_{\eps h}^{-1}$ by suitably exploiting the coercivity of 
the elastic energy. This would however require to strengthen the corresponding growth assumptions. 
\end{remark}

\section*{\UUU Acknowledgements} The authors acknowledge the hospitality and 
the support of the Erwin Schr\"odinger Institute of the University of Vienna, 
where most of this research has been performed, \UUU as well as the partial 
support from Austrian-Czech projects {16-34894L} (FWF/CSF) and 
7AMB16AT015 (FWF/MSMT CR). \EEE T.R acknowledges also the support of 
Czech Science Foundation (CSF) project 16-03823S 
and 17-04301S 
as well as through the institutional support RVO:\,61388998 (\v
CR). 
U.S.\ \UUU is partially supported by \EEE the Austrian Science Fund (FWF)
 projects  F\,65,  P\,27052, and I\,2375 and \UUU by \EEE the Vienna Science and 
Technology Fund (WWTF) project  MA14-009.
\UUU The authors are gratefully indebted to the anaonymous referee for the 
careful reading of the manuscript and many valuable suggestions. \EEE

\bibliographystyle{plain}

\end{document}